\documentclass[11pt]{article}
\usepackage{amsmath}
\usepackage{amssymb}
\usepackage{amsthm}
\usepackage{enumerate}
\usepackage[usenames,dvipsnames]{pstricks}
\usepackage{epsfig}
\usepackage{pst-grad}
\usepackage{pst-plot} 
\usepackage[a4paper]{geometry}
\theoremstyle{plain}
\newtheorem{thms}{Theorem}
\newtheorem{thm}{Theorem}
\newtheorem{lem}{Lemma}
\newtheorem{prop}{Proposition}

\theoremstyle{definition}
\newtheorem{defi}{Definition}
\newtheorem{ass}{Assumption}
\theoremstyle{plain}
\newtheorem{rmq}{Remark}

\newcommand{\N}{\ensuremath{\mathbb{N}}}
\newcommand{\eps}{\ensuremath{\varepsilon}} 
\newcommand{\Op}{\ensuremath{\mathcal{L}}}
\newcommand{\Hyp}{\ensuremath{\mathbb{H}}}
\newcommand{\R}{\ensuremath{\mathbb{R}}}
\newcommand{\Prob}{\ensuremath{\mathbb{P}}}

\newcommand{\E}{\ensuremath{\mathbb{E}}}

\makeatletter
\newcommand*\rel@kern[1]{\kern#1\dimexpr\macc@kerna}
\newcommand*\widebar[1]{%
  \begingroup
  \def\mathaccent##1##2{%
    \rel@kern{0.8}%
    \overline{\rel@kern{-0.8}\macc@nucleus\rel@kern{0.2}}%
    \rel@kern{-0.2}%
  }%
  \macc@depth\@ne
  \let\math@bgroup\@empty \let\math@egroup\macc@set@skewchar
  \mathsurround\z@ \frozen@everymath{\mathgroup\macc@group\relax}%
  \macc@set@skewchar\relax
  \let\mathaccentV\macc@nested@a
  \macc@nested@a\relax111{#1}%
  \endgroup
}
\makeatother

\begin{document}
\title{Lyapunov spectrum of a relativistic stochastic flow in the Poincar\'e group.}
\author{Camille Tardif \footnote{\texttt{camille.tardif@uni.lu} Universit\'e du Luxembourg.} \\  }

\maketitle
\abstract{ We determine the Lyapunov spectrum and stable manifolds of some stochastic flows on the Poincar\'e group associated to Dudley's relativistic processes.}

\vspace{1cm}
 L\'evy processes in Lie groups. Poincar\'e group. Lyapunov spectrum. Hyperbolic dynamics.

\tableofcontents

\section{Introduction}
In 1966 Dudley \cite{Dud66} defined a class of relativistic processes with Lorentzian-covariant dynamics in the framework of special relativity. Such  a process $\xi_t$ with values in Minkowski space-time $\R^{1,d}$, is differentiable and has velocity smaller than the speed of light. So it can be parametrized by its proper time, which amounts to impose to the velocity $\dot{\xi}_t$ to be an element of the unit pseudo sphere $\Hyp^d$of $\R^{1,d}$. The restriction to the tangent  space of $\Hyp^d$ of Minkowski ambient pseudo-metric turns $\Hyp^d$ into a Riemannian manifold of constant negative curvature. The invariance of the process $(\dot{\xi}_t, \xi_t) $ by the natural action of the set of Lorentz transforms on $\Hyp^d \times \R^{1,d}$ imposes to the laws of  $\dot{\xi}_t$  to be invariant by the action of the isometries of $\Hyp^d$. Among this class of relativistic processes, there is essentially only one which is continuous. It corresponds to the case where $\dot{\xi}_t$ is a Riemannian Brownian motion in the hyperbolic space and in this case $(\dot{\xi}_t, \xi_t)$ is called \emph{Dudley diffusion}. Forty years after this seminal work, Franchi and Le Jan \cite{FLJ07} extended Dudley diffusion to the framework of any Lorentz manifold. They defined relativistic processes with Lorentzian-covariant dynamics on generic Lorentzian manifolds by rolling without slipping a Dudley diffusion on the unit tangent space. They studied the asymptotic behavior of such diffusion in the Schwarzschild space-time. Bailleul \cite{Bail08} succeeded to compute the Poisson boundary of Dudley diffusion in  Minkowski space-time and showed that it coincides with the causal boundary of $\R^{1,d}$. The asymptotic behavior of relativistic diffusions was investigated in other non flat Lorentzian manifolds (\cite{Angst09}, \cite{Franchi09}) with the aim of  describing how the asymptotic behavior of the diffusion reflects the asymptotic geometry of the manifold. 

In this work we ask a new question concerning these processes dealing with the asymptotic behaviour of some stochastic flow associated to it. As Brownian motion on a Riemannian manifold , the relativistic diffusion \cite{FLJ07}  is obtained by projecting a diffusion process with values in the orthonormal frame bundle, solution of a stochastic differential equation. This SDE generates a stochastic flow which, in our Lorentzian framework, consists in a stochastic perturbation of the geodesic flow. Existence and computation, for example, of the Lyapunov spectrum and stable manifolds of these flows may be investigated in the same way as it was done by Carverhill and Elworthy \cite{Carv/Elw} for the canonical stochastic flow in the Riemannian framework. The main difficulty to study the flow of relativistic processes comes from the fact that the orthonormal frame bundle of a Lorentz manifold is never compact. Nevertheless in this article we provide a study of the asymptotic dynamics of the stochastic flow generated by Dudley processes in  the Minkowski space-time (without restricting ourselves to the diffusion case). Precisely, in this framework, the orthonormal frame bundle is identified with the Poincar\'e group $\widetilde{G}:= PSO(1,d) \ltimes \R^{1,d}$ and  denoting $\phi_t$ the left invariant stochastic flow associated to one of Dudley's processes in $\widetilde{G}$ we obtain the description of the Lyapunov spectrum and the stable manifolds of $\varphi_t$. Precisely we obtain the following two results.

\begin{thms}[Lyapunov spectrum]
There exist a constant $\alpha >0 $ and two asymptotic random Lie sub-algebras $V_\infty^- \subset V_\infty^0$ of $\mathrm{Lie}(\widetilde{G})$ such that  for some norm $\Vert \cdot \Vert$ on  $\mathrm{Lie}(  \widetilde{G} )$ and  $\widetilde{X} \in \mathrm{Lie}(  \widetilde{G} )$ we have  for almost every trajectory
\[
\frac{1}{t} \log \Vert  d \varphi_t (\mathrm{Id})( \widetilde{X})  \Vert_{\varphi_t(\mathrm{Id})} \underset{t \to +\infty}{\longrightarrow} \left \{  \begin{matrix} \alpha & \mathrm{if} &  \widetilde{X} \in  \mathrm{Lie}(  \widetilde{G} ) \setminus V_\infty^0 \\ 0 & \mathrm{if} &  \widetilde{X} \in  V_\infty^0 \setminus V_\infty^-  \\ -\alpha & \mathrm{if} &  \widetilde{X} \in  V_\infty^- \setminus \{ 0 \}  \end{matrix}\right.
\]
\end{thms}

\begin{thms}[Stable manifolds]
Denote by  $\mathcal{V}_{\infty}^{-}:=\exp ( V_\infty^-)$ and $d$  the distance associated to a left invariant and $\mathrm{Ad}(SO(d))$-invariant Riemanian metric in $\widetilde{G}$. Then for any two distinct points $\tilde{g}'$ and $\tilde{g}$  in $\widetilde{G}$ we have
 \begin{itemize}
\item If $\tilde{g}' \in \tilde{g} \mathcal{V}_{\infty}^{-}$ then 
\[
\frac{1}{t} \log d \left (  \varphi_t (\tilde{g}) , \varphi_t (\tilde{g}') \right ) \underset{ t \to + \infty }{\longrightarrow} -\alpha.
\]
\item If $\tilde{g}' \notin \tilde{g} \mathcal{V}_{\infty}^{-}$ then
\[
\liminf_{t \to \infty} d\left (  \varphi_t (\tilde{g}) , \varphi_t (\tilde{g}') \right ) >0.
\]
\end{itemize}

\end{thms}

We begin by constructing, in section \ref{Dudproc}, Dudley processes as  projections of left L\'evy processes on the Poincar\'e group $\widetilde{G}$, identified with the orthonormal frame bundle of the Minkowski space-time. These L\'evy processes are solutions of stochastic integral equations and induce a left invariant stochastic flow $\varphi_t$ in $\widetilde{G}$. In section \ref{asympt} we find the asymptotic behavior of Dudley processes and exhibit the asymptotic random variables $(\theta_\infty, \lambda_\infty) \in \mathbb{S}^{d-1} \times \R^{*}_{+}$. Finally in section \ref{LyapSpec}  we prove Theorems \ref{Lyap} and \ref{stable} and explicit the projection of the stable manifold in $\Hyp^d \times \R^{1,d}$ by showing that it corresponds to a skew product of a horosphere by a line.

Note that stochastic flows generated by L\'evy processes on semi-simple Lie groups were intensively studied by Liao (\cite{Liao02}, \cite{Liao01}, \cite{Liao04}). But his results cannot be used directly here since our L\'evy processes lie in the Poincar\'e group which is not semi-simple. Moreover in our work we suppose only that the Levy measure is integrable at infinity whereas Liao  \cite{Liao04} request the entire integrability of it. 
 
Our work is also strongly inspirited by the work of Bailleul and Raugi \cite{Bail/Raug} where the authors used Raugi's methods \cite{Raugi77} to find the Poisson boundary of Dudley diffusion.


\section{Dudley processes and their lift in the Poincar\'e group}\label{Dudproc}

We present in this section the geometrical framework of special relativity and define a natural  class of relativistic Markov processes with Lorentzian-covariant dynamics introduced by Dudley in \cite{Dud66}. They are obtained by projecting  left L\'evy processes with values in the Poincar\'e group and are described by two parameters: a diffusion coefficient $\sigma \in \R$ and a jump intensity L\'evy measure $\nu$ on $\R^*_+$.

\subsection{ Minkowski space-time and Poincar\'e group} The Minkowski space-time  $\R^{1,d}$ is  $\R \times \R^d$ endowed with the Lorentz quadratic form $q$ defined by 
\[
\forall \xi= (\xi^0, \xi^1, \dots, \xi^d) \in \R \times \R^d , \quad \quad q(\xi) = \left (\xi^{0} \right )^{2} -  \left (\xi^{1} \right )^{2} - \cdots  -  \left (\xi^{d} \right )^{2}.
\]
We denote by $\vec{\xi}:= ( \xi^1, \dots, \xi^d )^t$ the space component of $\xi$.

Set 
\[
Q = \mathrm{Diag} (1, -1, \dots, -1 )
\]
the matrix of $q$ in the canonical basis $(e_0, e_1, \dots, e_d)$. Time orientation is given by the constant vector field $e_0$ and some $\xi \in \R^{1,d}$ is said to be future oriented when $q( \xi , e_0) >0 $. A path $\gamma_s$ in $\R^{1,d}$ is said to be \emph{time-like} when it is differentiable almost everywhere and $q(\dot{\gamma}_s) >0 $ and $q(\dot{\gamma}_s, e_0) >0$. 
The Poincar\'e group is the group of affine $q$-isometries which preserve orientation and time-orientation. It is the semi-direct product connected  group \[ \widetilde{G} := PSO(1,d) \ltimes \R^{1,d} \] where $G:=PSO(1,d)$ denotes the group of linear $q$-isometries which preserve orientation and time-orientation. An element $\tilde{g}= (g, \xi) \in \widetilde{G}$ is made up of its linear part $g \in G$ and its translation part $\xi$. We identify $G$ with the sub-group of $\widetilde{G}$ which fixes $0$. By this way, we identify $\R^{1,d}$ with the homogeneous space $\widetilde{G}/G$. The identity element of $\widetilde{G}$ and $G$ is denoted by $\mathrm{Id}$ (thus for us $\mathrm{Id}= (\mathrm{Id},0)$).  At $\tilde{g} =( g, \xi ) \in \widetilde{G}$ we associate the affine frame $ \left ( (g(e_0), g(e_1), \dots, g(e_d)) ; \xi \right ) $ of $\R^{1,d}$ and $\widetilde{G}$ is identified with the bundle of $q$-orthonormal, oriented and time-oriented, frames over $\R^{1,d}$. We denote by
\[
\begin{matrix}
\tilde{\pi}: & \widetilde{G} & \longrightarrow & \R^{1,d}  \\
& \tilde{g}=(g, \xi) & \longmapsto & \xi
\end{matrix}
\]
the projection associated to this trivial fibration and  $G =\tilde{\pi}^{-1} \{  0 \}$. The canonical basis being fixed we identify $G$ with the matrix group
\[
G= \left \{ g \in \mathrm{SL}(\R^{d+1} ), \ g Q g^t =  Q , \  q(g(e_0), e_0 ) >0  \right \},
\]
and its Lie algebra is
\begin{align*}
\mathrm{Lie}(G) &= \left \{ X \in \mathcal{M}_{d+1}(\R),  \quad  X Q  -Q X^t = 0 \right \} \\
&= \left \{ \left (  \begin{matrix} 0 & b^t \\ b & C  \end{matrix} \right ),  \  \ b \in \R^{d} , \  C \in \mathcal{M}_{d}(\R) \ \text{s.t} \ C=-C^t  \right \}.
\end{align*}
We have $\mathrm{Lie} (\widetilde{G}) = \mathrm{Lie}(G) \times \R^{1,d}$ and for $\widetilde{X}= (X, x) , \widetilde{Y}=(Y,y) \in \mathrm{Lie} (\widetilde{G}) $
\[
[\widetilde{X}, \widetilde{Y}] = ( [X,Y], Xy-Yx).
\]
We identify $\mathrm{Lie}(G)$ with  $\mathrm{Ker}(d_{\mathrm{Id}} \tilde{\pi} )$ and its elements are vertical for the fibration $\tilde{\pi}$. 
We set 
\begin{align*}
V_i & := e_0 e_i^{t} + e_i e_0^{t} \quad i=1, \dots, d \\
V_{ij} & := [V_i, V_j]= e_i e_j^{t} -e_j e_i^{t} \quad  j>i.
\end{align*}
Moreover we set
\[
H_0:= (0, e_0 ) \in \mathrm{Lie}(\widetilde{G})
\]
which is horizontal for the fibration $\tilde{\pi}$.

{\bf Notation} For $\widetilde{X} \in \mathrm{Lie}(\widetilde{G})$ we denote by $\widetilde{X}^{l}$ the left invariant  vector field in $\widetilde{G}$ associated.

Denote by $K$ the subgroup of $G$ made of the rotations of $\R^d$. We have
\[
K := \left \{  \left ( \begin{matrix} 1 & 0 \\ 0 & k  \end{matrix} \right ), k \in SO(d)  \right \},
\] 
and $K$ is also the stabilizer of $e_0$ under the action of $G$ on $\R^{1,d}$. The homogeneous space $G/K$ can be identified with the orbit of $e_0$ under the action of $G$ which is  the unit pseudo sphere $\Hyp^d := \{ \xi \in \R^{1,d}, \ q(\xi)=1, \xi^0 >0 \}$ and is a Riemannian manifold of constant negative curvature when its tangent space is endowed with the restriction of $q$ on it.

For $r \in \R^{+}$ and $ \theta \in \mathbb{S}^{d-1} \subset \R^{d}$, define 
 \[
 S(r, \theta):= \exp \left ( r \sum_{i=1}^d \theta^i V_i \right ) = \left ( \begin{matrix} \cosh(r) & \sinh(r) \theta^{t}  \\ \sinh(r) \theta & \mathrm{Id} + ( \cosh(r) -1 ) \theta \theta^{t} \end{matrix}  \right ).
 \] 
Each $g \in G$ can be decomposed in \emph{polar form} $g= S(r(g), \theta(g) )R$ where $R\in K$.

\subsection{ Dudley processes}

In this paragraph we define the relativistic processes introduced by Dudley in \cite{Dud66}. These processes enjoy two natural properties:
\begin{itemize}
\item they are $\widetilde{G}$-invariant i.e their dynamics are invariant by a change of $q$-orthonormal frame   
\item  their trajectories in $\R^{1,d}$ are time-like: they are almost everywhere differentiable,and the tangents vectors are time-like and time oriented.
\end{itemize}

First remark that no Markov processes with values in $\R^{1,d}$ is $\tilde{G}$-invariant. Indeed, the law at some time $t>0$  of such process starting at $0$ would be a $G$-invariant probability measure in $\R^{1,d}$ which is necessary trivial by the following lemma.

\begin{lem}
The only $G$-invariant probability measure in $\R^{1,d}$ is the Dirac measure at $0$.
\end{lem}
\begin{proof}
Let $\mu$ be a $G$-invariant probability measure in $\R^{1,d}$. First suppose that the support of $\mu$ is not contained in the $q$-orthogonal hyperplane of some light-like line $\{ u (e_0  + \hat{\theta}), u \in \R \}$ ( $\hat{\theta}:= \sum_{i=1}^d \hat{\theta}^i e_i \in \mathbb{S}^{d-1} $ ). So there exist a compact set $C$ in the complement of this hyperplane such that $\mu( C) >0$. For $g= S(r, \hat{\theta})$, $r>0$ and $\xi= (\xi^0, \vec{\xi} ) \in \R^{1,d}$, denoting $\Vert \cdot \Vert$ the Euclidean norm in $\R^{1,d}$ we have 
\begin{align*}
 \Vert g(\xi) \Vert^2&= 2q( g(\xi), e_0)^2  -q(g(\xi) )= 2 q(\xi, g^{-1}(e_0) )^2 -q(\xi)
 =  2 \left ( \cosh(r) \xi^0 - \sinh(r)   \hat{\theta} \cdot \vec{\xi} \right )^2 - q(\xi) \\
 &= 2 \left (  \cosh(r) ( \xi^0 - \hat{\theta} \cdot \vec{\xi} )  + e^{-r} \hat{\theta} \cdot \vec{\xi} \right )^2 -q(\xi) \\
 &= 2 \left (  \cosh(r) q( \xi, e_0 + \hat{\theta} ) + e^{-r} \hat{\theta} \cdot \vec{\xi}   \right )^2 -q(\xi).
\end{align*}

Since $C$ is a compact set in the complement of the  $q$-orthogonal hyperplan to $\{ u (e_0  + \hat{\theta}), u \in \R \}$ then $\inf_ {\xi \in C} \vert q( \xi, e_0 + \hat{\theta} )  \vert > 0$ and thus $r$ can be chosen such that $\inf_{\xi \in C } \Vert g(\xi) \Vert $ is arbitrary large. Thus $g$ can be chosen such that $C$ and $g(C)$ are disjoint. Furthermore  $q( g(\xi), e_0 + \hat{\theta} )= q (\xi , g^{-1}(e_0 + \hat{\theta} ) )= e^r q(\xi, e_0 + \hat{\theta})$ and thus $g(C)$ belongs to the  complement  of the hyperplan $q$-orthogonal to $e_0 + \hat{\theta}$. By iteration we can find a sequence $g_k = S(r_k, \hat{\theta})$, $k \in \N$ such that the compact sets $g_k (C)$ are pairwise disjoint. Thus we obtain a contradiction writing $1 \geq \mu ( \cup_k g_k(C)) = \sum_k \mu ( g_k (C) ) = \sum_k  \mu(C)  = + \infty$.

Now if the support of $\mu$ is contained in some hyperplan tangent to the light-cone and is not restricted to $0$, we can find a compact set $C$ in this hyperplan with $0 \notin C$ and $\mu(C) >0$ and we can choose a rotation $R \in K$ such that $R(C)$ is not in the hyperplan. Thus $\mu(C)= \mu(R(C) )= 0$ and it gives a contradiction. So we proved that $\mu$ is necessary the Dirac measure in $\{0 \}$. 
\end{proof}

Thus $\R^{1,d}= \widetilde{G}/G$ cannot be the space of states of some non-trivial $\widetilde{G}$-invariant Markov process. But $\widetilde{G}$-homogeneous spaces of the form $\widetilde{G}/ \widetilde{K}$ where $\widetilde{K}$ is a compact subgroup of $\widetilde{G}$ have some $\widetilde{K}$-invariant probability measure and may be endowed with some $\widetilde{G}$-invariant Markov processes (see \cite{Liao04} or \cite{Hunt56} ). The smallest spaces of states we can consider correspond to the maximal compact sub-group of $\widetilde{G}$.  Thus it is natural to consider the space of states $\widetilde{G}/K  \simeq \Hyp^d \times \R^{1,d}$. The group $K$ is seen as the subgroup of $\widetilde{G}$ which stabilize $0$ and $e_0$ under the action of $\widetilde{G}$ on $\R^{1,,d}$.



We denote by $\pi : \widetilde{G} \longmapsto \Hyp^d \times \R^{1,d} \simeq \widetilde{G}/K$ the canonical projection  \[\forall \tilde{g}= (g, \xi) \in \widetilde{G} \quad \pi(\tilde{g})=(g(e_0), \xi ). \]

The following Proposition exhibit all the relativistic processes in $\Hyp^d \times \R^{1,d}$. It is essentially an application of a result of Liao ( Theorem 2.1 and 2.2 p 42 in \cite{Liao04} ).

\begin{prop}\label{defiDudley}
The Markov processes on $\Hyp^d \times \R^{1,d}$, starting at $(\zeta_0, \xi_0)$, which are $\widetilde{G}$-invariant and whose trajectories are time-like are of the form $(\zeta_s, \xi_s)$ where $ \zeta_s$ is a $G$-invariant Markov process on $\Hyp^{d}$ and $\xi_t = \xi_0 + a\int_{0}^{t} \zeta_s  ds $; $a$ being some positive constant.  For  such a process there exist $\sigma >0$ and a measure $\nu$  on $\R^+$ satisfying 
\[
\int_0^{+ \infty} \min(1,r^2) \nu(dr) < + \infty,
\]
such that $(\zeta_t, \xi_t) = \pi (\tilde{g}_t )$ in law where $\tilde{g}_t$ is a left Levy process on $\widetilde{G}$ starting at $\tilde{g}_0$ s.t $(\zeta_0, \xi_0) = \pi (\tilde{g}_0 )$ of which  generator $\tilde{\Op}$  is defined  by
\begin{align*}
\forall f \in C^2( \widetilde{G} ) \quad \widetilde{\Op} f (\tilde{g}) = a H_0^{l}f(\tilde{g}) &+ \frac{\sigma^2}{2} \sum_{i=1}^{d} (V_{i}^l)^2 f (\tilde{g})  \\
 &+ \int_0^{+\infty } \int_{\mathbb{S}^d} \left ( f(\tilde{g}S(r, \theta) ) -f(\tilde{g}) - r \mathbf{1}_{r \in [0,1] }\sum_{i=1}^d \theta^{i} V_i^l f(\tilde{g}) \right ) \nu(dr) d \theta.
\end{align*} 
\end{prop}

\begin{defi}[Dudley processes]
When $a=1$ then $\dot{\xi}_t = \zeta_t$ and $\xi_t$ is parametrized by its \emph{proper time}, i.e
$
q( \dot{\xi}_t )=1.
$
When moreover $\sigma$ or $\nu$ is non trival we call $(\dot{\xi}_t , \xi_t)$ a Dudley process and we consider exclusively these processes in the sequel. When $\nu =0$ $(\dot{\xi}_t, \xi_t)$ is continuous and is called \emph{Dudley diffusion}.
\end{defi}
\begin{rmq}
The process $\xi_t$ is differentiable and $\dot{\xi}_t$ is c\`adl\`ag.
\end{rmq}

\begin{proof}[Proof of Proposition \ref{defiDudley}]
Let $(\zeta_t, \xi_t)$ a Markov (Feller) process on $\Hyp^{d} \times \R^{1,d}$ starting at $(\zeta_0, \xi_0)$, which is $\widetilde{G}$-invariant and whose trajectories are time-like. By choosing $\tilde{g}_0$ such that $(\zeta_0, \xi_0) = \pi (\tilde{g}_0 )$ and considering the Markov process $\tilde{g}_0^{-1} (\zeta_t, \xi_t)$ it remains to prove the proposition in the case where $(\zeta_0, \xi_0)= (e_0, 0)$. 

Set \[ H_i := ( 0, e_i ) \in \mathrm{Lie}(\widetilde{G}), \quad i=1, \dots, d.\] The family $\{ H_0, H_i, V_i,  V_{ij} \}_{i<j \in \{1, \dots, d \}}$ form an orthonormal basis of $\mathrm{Lie}(\widetilde{G})$ for an $Ad(K)$-invariant inner product on $\mathrm{Lie}( \widetilde{G} )$ and $\mathrm{Ker} \ d_{\mathrm{Id}} \pi  = \{  V_{ij} \}_{i<j \in \{1, \dots, d\} }$. 

We set \[ X_0 := H_0 \quad X_i := H_i \quad X_{d+i}:= V_i \quad i=1, \dots, d. \] An $\tilde{h}=(h, \xi) \in \tilde{G}$ can be decomposed in $ \tilde{h}= \exp \left ( \sum_{i=0}^d x^{i}(\tilde{h}) X_i \right ) \exp \left ( \sum_{i= d+1}^{2d} x^{i}(\tilde{h}) X_i \right )R$ where $R \in K$, $x^{0}(\tilde{h})= \xi^{0}$ and  for $i=1, \dots, d$  $x^{i}(\tilde{h}) = \xi^{i}$ and $x^{d+i}(\tilde{h})= r(h)\theta^{i}(h)$.  By Theorem 2.1 and 2.2 p 42 of \cite{Liao04}, $(\zeta_t, \xi_t)$ coincide in law with $\pi(\tilde{g}_t)$ where $\tilde{g}_t$ is a left Levy process in $\tilde{G}$, starting at $\mathrm{Id}$, which is $K$-right invariant and generated by
\[
\widetilde{\Op} f(\tilde{g}) = \frac{1}{2}\sum_{i, j =0}^{2d } a_{ij} X_i^{l} X_j^{l}f(\tilde{g}) + \sum_{i=0}^{2d} b^{i} X_i^{l}f(\tilde{g}) + \int_{\widetilde{G}} \left ( f(\tilde{g} \tilde{h}) -f(\tilde{g}) - \sum_{i=0}^{2d} \mathbf{1}_{\underset{\xi^{i}(h) \leq 1}{r(h) \leq 1 }  }x^{i}(\tilde{h}) X_{i}f(\tilde{g})   \right ) \widetilde{\Pi}(d \tilde{h}).
\]
The matrice $A:=(a_{ij})$ is a positive symmetric , $(b^{i})_i \in \R^{2d+1}$ and $\widetilde{\Pi}$ is a Levy measure invariant  by $K$-conjugation in $\widetilde{G}$.  The right $K$-invariance of $\tilde{g}_t$ ensures that 
for all $k \in K \simeq SO(d)$ , $\mathrm{diag}(1,k, k) A \mathrm{diag}(1,k, k)^{-1} = A$.  Thus, $A$ is necessarily of the form
\[
\left (
\begin{matrix}
\hat{\sigma}  & 0 & 0 \\
0 & \tilde{\sigma} \mathrm{Id}&  \sigma' \mathrm{Id} \\
0 & \sigma' \mathrm{Id}& \sigma \mathrm{Id}
\end{matrix} \right ),
\]
where $\hat{\sigma}, \tilde{\sigma} \geq 0$ and $\tilde{\sigma} \sigma \geq (\sigma ')^2$. Moreover, using again $K$-invariance  it comes $b^i =0$ for $i=1, \dots, 2d$ and we set $a:= b^{0}$.  The trajectories of $\tilde{g}_t$ projected in $\R^{1,d}$ need to be differentiable so the jump measure $\Pi$ is supported on $G$ and $\hat{\sigma}$ and $\tilde{\sigma}$ are necessarily null. Since the trajectories are time-oriented we have $a >0$. The push forward of $\widetilde{\Pi}$ by $\pi$ is supported on $G/K \simeq \Hyp^{d}$ and is $K$-invariant. Thus $\Pi$ can be chosen of the form
\[
\forall f \in C_0(G) \quad \Pi f =  \int_0^{+\infty} \int_{\mathbb{S}^{d-1}}  f(S(r, \theta)) \nu(dr)  d\theta
\] 
where $\nu$ is a Levy measure on $\R^{*}_{+}$ ( i.e satisfying $\int \min (1,r^2) \nu(dr) <+\infty$). 
\end{proof}

 Denote by $g_t$ the $G$-component of $\tilde{g}_t$. Thus $g_t (e_0) = \dot{\xi}_t$ and $ \xi_t = \int_0^t g_s (e_0) ds $. By definition $g_t$ is a $G$-valued left Levy process, $K$-right invariant, generated by $\Op$ defined by
 \[
 \forall f \in C^2(G) \quad \Op f(g) = \frac{\sigma^2}{2} \sum_{i=1}^{d} (V_{i}^l)^2 f (g) + \int_0^{+\infty } \int_{\mathbb{S}^d} \left ( f(\tilde{g}S(r, \theta) ) -f(\tilde{g}) - r \mathbf{1}_{r \in [0,1] }\sum_{i=1}^d \theta^{i} V_i^l f(\tilde{g}) \right ) \nu(dr) d \theta.
 \]
Denote by $\Pi$ the Levy measure supported on $G$ defined by 
 \[
 \Pi f = \int_0^{+\infty} \int_{\mathbb{S}^{d-1}} f(S(r, \theta) ) \nu(dr) d\theta.
 \]
 
 Define $U_0 := \{ g \in G, r(g) \leq 1\}$ which is a $K$ invariant neighborhood of $\mathrm{Id}$ in $G$.  For $f \in C^2(G)$ we have the following It\^o formula (see \cite{App/Kun}) for $g_t$
\begin{align} \label{Ito}  \notag
f(g_t)& = f(\mathrm{Id}) + \sigma  \sum_{i=1}^d\int_0^t V_i^{l}f(g_{s^{-}}) dB^{i}_s + \frac{\sigma^2}{2}  \int_0^t  \sum_{i=1}^{d} (V_i^{l})^2f(g_{s^{-}}) ds +\int_0^t \int_{U_0}  \left ( f(g_{s^{-}}h) -f(g_{s^{-}}) \right ) \tilde{N}(ds,dh) \\ 
&+ \int_0^t  \int_{U_0} \left ( f(g_{s^{-}}h) -f(g_{s^{-}}) -r(h) \sum_{i=1}^{d}  \theta^{i}(h) V_i^l f(g_{s^{-}}) \right )  ds \Pi(dh) \\ 
& + \int_0^t \int_{(U_0)^c} \left ( f(g_{s^{-}}h) -f(g_{s^{-}}) \right ) N(ds,dh), \notag
\end{align}
where $B_t$ is a Brownian motion of  $\R^d$, $N$ is a  Poisson random measure on $\R \times G$ of intensity measure $dt \otimes \Pi$ and $\tilde{N}(ds, dh) := N(ds, dh) - ds \Pi(dh) $ is the compensated random measure associated. 
\section{Asymptotic random variables}\label{asympt}
In this section we determine the asymptotic behavior of $\pi(\tilde{g}_t)=(g_{t} (e_0), \xi_t)$  under an integrability condition on the jump intensity measure $\nu$ (Ass.\ref{integrability}). Writing $g_t=n_t a_t k_t$ in some Iwasawa decomposition of $G$ we first prove (Prop.\ref{alpha}), applying the It\^o formula \eqref{Ito} and the law of large number, that the abelian term $a_t = \exp( \alpha_t V_1) $ is \emph{positively contracting},  $\frac{\alpha_t}{t}$ converges almost surely to a positive constant $\alpha$ depending explicitly on $\sigma$ and $\nu$. Next  we prove (Prop.\ref{convht}) that the nilpotent term $n_t$ converges  almost surely to an asymptotic random variable $n_\infty$ and this convergence is exponentially fast with rate   $\alpha$. Then we investigate (Prop.\ref{xit}) the asymptotic behavior of $\xi_t$ in $\R^{1,d}$. Geometrically, seen in the projective space, the $\Hyp^d$-valued process $g_t(e_0)$ converges to a limit angle $\theta_\infty \in \partial \Hyp^d \simeq \mathbb{S}^{d-1}$  of which $n_\infty$ is a stereographic projection. Moreover, the process $\xi_t$ is asymptotic to some affine hyperplane $q$-orthogonal to $\theta_\infty$ of which position is fixed by another asymptotic random variable $\lambda_\infty \in \R^{*}_{+}$. Figure \ref{fig} sum up the asymptotic results.

\begin{figure}[h!]
 \centering
\input{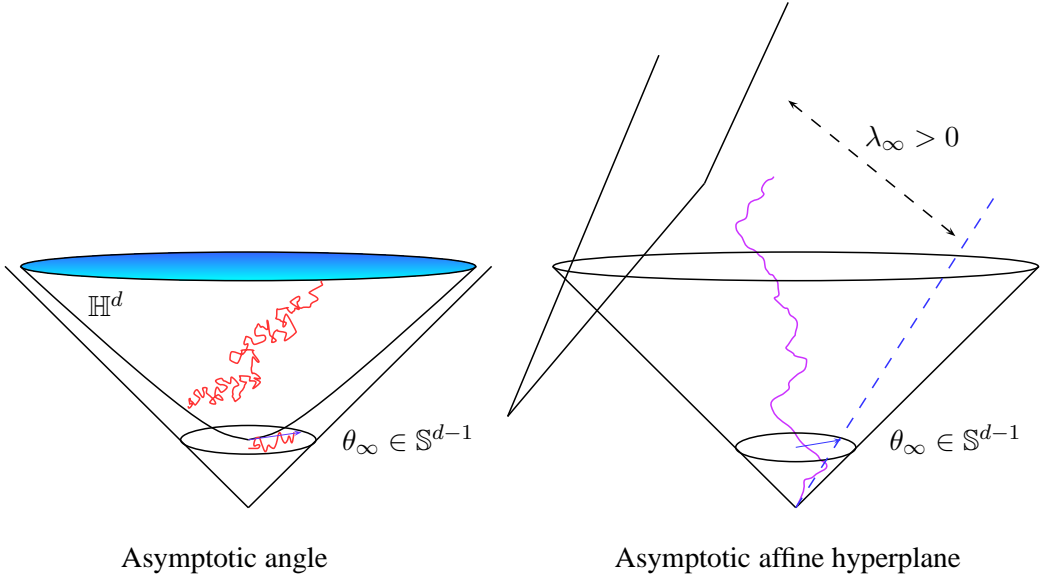}
\caption{Asymptotic behavior of a Dudley diffusion } \label{fig}
\end{figure}

\subsection{Iwasawa decomposition in $G$}

Although a polar decomposition of $G$ was used to introduce $g_t$ (defining the $K$ invariant measure $\Pi$),  Iwasawa decomposition seems to be more adapted to describe its asymptotic dynamics. Introduce briefly this decomposition.

The maximal abelian subalgebra contained in $\mathrm{Vect}\{ V_1, \dots, V_d \}$, which is the orthogonal subspace of $\mathrm{Lie}(G)$ of $\mathcal{K}$ for the Killing form,  is  of dimension one. Let choose $\mathcal{A} := \mathrm{Vect} \{ V_1 \}$ one of them. The linear endomorphism $\mathrm{ad}(V_1)$ of $\mathrm{Lie}(G)$  is diagonalisable with eigenvalues $-1, 0$ and $1$. Set
 \[
 \mathcal{\mathcal{N}}= \left \{ X \in \mathrm{Lie}(G), \  \mathrm{ad}(V_1) X= -X  \right \} \quad   \widebar{\mathcal{N}}=\left \{ X \in \mathrm{Lie}(G), \  \mathrm{ad}(V_1) X= X  \right \} 
  \]
  the eigenspace corresponding respectively to the eigenvalue $-1$ and $1$. Explicitly
\[
\mathcal{N}=\mathrm{Vect}\{ V_i - V_{1i}  , \ \  i=2, \dots, d \} \quad \mathrm{and} \ \ \widebar{\mathcal{N}}= \mathrm{Vect}\{ V_i + V_{1i}  , \ \  i=2, \dots, d \}.
\]
The eigenspace corresponding to $0$ is $\mathcal{A}\oplus \mathcal{M}$ where $\mathcal{M}$ is the sub algebra of elements of $\mathcal{K}$ which commute with the elements of $\mathcal{A}$. Explicitly
\[ 
\mathcal{M}= \mathrm{Vect}\{ V_{ij} ,\ \  i, j =2\dots d  \} =\left \{ \left ( \begin{matrix} 0 & 0 & 0 \\ 0 & 0 & 0 \\ 0 & 0 &  C    \end{matrix}  \right ), \ \  C \in \mathrm{so}(d-1)    \right \}. 
\]
The subspace $\mathcal{N}$ is a nilpotent Lie algebra (even abelian since $[ \mathcal{N}, \mathcal{N}]=0$). The corresponding Iwasawa decomposition of $\mathrm{Lie}(G)$ is 
\[
\mathrm{Lie}(G)= \mathcal{N}\oplus \mathcal{A} \oplus \mathcal{K}.
\]
For $X \in \mathrm{Lie}(G)$ we denote  by $\{ X \}_{\mathcal{N}}$ (resp. $\{ X \}_{\mathcal{A}}$ and $\{ X \}_{\mathcal{K}}$) its projection in $\mathcal{N}$ (resp. $\mathcal{A}$ and $\mathcal{K}$ ) thus $X= \{ X \}_{\mathcal{N}}+\{ X \}_{\mathcal{A}}+\{ X \}_{\mathcal{K}}$. 
Denoting by $A:=\exp ( \mathcal{A})$, $N:= \exp(\mathcal{N})$ the subgroup corresponding we obtain the corresponding Iwasawa decompositions of $G$
\[
G=NAK.
\]
Moreover, the mapping from $N \times A \times K$ to $G$ which maps $(n,a,k)$ to $nak$ is an analytic diffeomorphism. For $g\in G$ we denote by $g=(g)_{N} (g)_A (g)_K$ its decomposition in Iwasawa coordinates. To simplify  notations  set $n_t:=(g_t)_{N}$, $a_t:=(g_t)_A$ and $k_t=(g_t)_K$, thus $g_t=n_t a_t k_t$.

Note that we have other Iwasawa decompositions like $\mathrm{Lie}(G) = \widebar{N}\oplus \mathcal{A} \oplus \mathcal{K}$ (with $G= \widebar{N}AK$). 

\noindent\textbf{Iwasawa and polar coordinates}. \\
\noindent For $g \in G$ written in polar form $g = \exp \left (r \sum_{i=1}^{d} \theta^{i} \left ( V_i -V_{1i}\right ) \right )R$ where $R\in K $, we have $ q( e_0 , g(e_0) )= \cosh(r)$.

 Now denoting by $b \in \R^{d-2}$ and $u \in \R$ such that $(g)_N = \exp\left ( \sum_{i=2}^{d} b^i (V_i -V_{1i}  ) \right )$ and $(g)_A= \exp \left ( u V_1\right )$ we can compute explicitly $q( e_0 , g(e_0) )$ in terms of $b$ and $u$ and we obtain 
\begin{align}
\cosh(r) = \left ( 1+ \frac{\Vert b \Vert^2}{2} \right ) \cosh (u )  + \frac{\Vert b \Vert^2}{2} \sinh( u  ). \label{link}
\end{align}
Moreover $u$ can be expressed in term of $\theta$ and $r$ via 
\begin{align}
e^{u} =\cosh(r) + \theta^1 \sinh(r) \label{link2}
\end{align}
\\

\subsection{Asymptotic behavior of the $G$-component}

The aim of this section is to show that $n_t$ converges almost surely to an asymptotic $N$-valued random variable $n_{\infty}$ and that the convergence is exponentially fast. This result, stated in Proposition \ref{convht}, appears to be a consequence of the contracting property of $a_t$. For this we need the following  integrability condition on $\nu$. The group $G$ is semi simple and the tools used in this section are very closed from those of Liao \cite{Liao04}.  Nevertheless we present a self-contained proof in our specific framework and our results are established under a weaker assumption than the ones of \cite{Liao04} ( see remark  \ref{rmq} ). Namely we suppose that the following integrability condition is satisfied. 

\begin{ass}\label{integrability} 

\[
\int_1^{+\infty} r \nu (dr) < +\infty.
\]
\end{ass}

The following proposition computes explicitly the linear drift of $a_t$ which appears to be positive. The proof is essentially a consequence of the law of large number. 

\begin{prop}\label{alpha}
Let denote by $\alpha_t$ the $\R$-valued process such that $a_t= \exp ( \alpha_t V_1)$. Then the following convergence holds almost surely
\[
\frac{\alpha_t}{t} \underset{t \to \infty}{\longrightarrow} \alpha >0.
\]
The positive constant $\alpha$ is 
\[
\alpha := \frac{d-1}{2} \sigma^2 + \int_0^{+\infty} \frac{r \cosh(r) -\sinh(r)}{\sinh(r)} \nu(dr).
\]
\end{prop}
\begin{proof}
First define $\log : A \to \mathcal{A} \simeq \R \ \exp(u V_1) \mapsto u V_1 $ and  apply  It\^o formula \eqref{Ito} to the smooth map $f: g \mapsto \log (g)_A$. Remark that for $g, h \in G$ $(gh)_A = (g)_A \left ( (g)_K h \right )_A$ and 
\[
f(g_{s-}h)-f(g_{s-})=\log \left ( k_{s-}g \right )_{A}.
\]
Moreover
\begin{align*}
V_{i}^{l}f(g_{s-})&=  \left. \frac{d}{dt} \log (g_{s-} e^{tV_{i}})_{A} \right \vert_{t=0} = \left. \frac{d}{dt}  \log \left (e^{tAd(k_{s-})V_{i}} \right )_{A}  \right \vert_{t=0} = \{ Ad(k_{s-})V_{i} \}_{\mathcal{A}} \\
(V_{i}^{l})^{2}f(g_{s-})&= \left. \frac{d}{dt}  \{ Ad\left (g_{s-}e^{tV_{i}} \right )_{K}V_{i} \}_{\mathcal{A}} \right \vert_{t=0}= \left. \frac{d}{dt}  \{ Ad\left (e^{tAd(k_{s-})V_{i}} \right )_{K}Ad(k_{s-})V_{i} \}_{\mathcal{A}} \right \vert_{t=0} \\&= \left \{\left [ \left\{ Ad(k_{s-})V_{i} \right\}_{\mathcal{K}}, Ad(k_{s-})V_{i} \right ] \right \}_{\mathcal{A}}.
\end{align*}
For $k\in K$  we have  $Ad(k)V_{i}= \sum_{j=1}^{d} k_{ij}V_{j}$ where $k=(k_{ij}) \in K \simeq SO(d)$. Then we compute
\begin{align*}
\sum_{i=1}^{d} \left \{\left [ \left\{ Ad(k_{s-})V_{i} \right\}_{\mathcal{K}}, Ad(k_{s-})V_{i} \right ] \right \}_{\mathcal{A}} &=\sum_{i=1}^{d} \sum_{j=1}^{d} \sum_{l=1}^{d} (k_{s-})_{ij}(k_{s-})_{il} \left \{ \left [ \{  V_{j} \}_{\mathcal{K}}, V_{l}  \right ] \right \}_{\mathcal{A}},
\end{align*}
 since $\{V_{j} \}_{\mathcal{K}}= V_{1j}$ (and 0 for $j=1$ ) and  $[V_{1j},V_{l}]=V_{1}\delta_{jl}$ we get
 \begin{align*}
\sum_{i=1}^{d} \left \{\left [ \left\{ Ad(k_{s-})V_{i} \right\}_{\mathcal{K}}, Ad(k_{s-})V_{i} \right ] \right \}_{\mathcal{A}} = \sum_{j=2}^{d} \sum_{i=1}^{d} ((k_{s-})_{ij})^{2} V_{1}= (d-1)V_{1}. 
\end{align*}
Thus It\^o formula \eqref{Ito} can be written 
\begin{align*}
\log(a_{t})&= M_{t} + \frac{d-1}{2}\sigma^{2}t + \int_{0}^{t} \int_{U_{0}^{c}} \log (k_{s-}h)_{A} N(ds,dh) \\& + \int_{0}^{t} \int_{U_{0}} \left ( \log(k_{s-}h)_{A} -r(h)\sum_{i=1}^{d} \theta^{i}(h) \{ Ad(k_{s-})V_{i} \}_{\mathcal{A}} \right ) ds \Pi(dh),
\end{align*}
where $M_{t}:= \sigma \sum_{i=1}^{d}\int_{0}^{t} \{ Ad(k_{s-})V_{i} \}_{\mathcal{A}}dB^{i}_{s} + \int_{0}^{t}\int_{U_{0}} (k_{s-}h)_{A} \tilde{N}(ds,dh) $ is a martingale. Its bracket is  $ \sigma^{2} \int_{0}^{t}  \sum_{i=1}^{d}\Vert \{ Ad(k_{s-})V_{i} \}_{\mathcal{A}} \Vert^{2} ds +  \int_{0}^{t} \int_{U_{0}} \Vert \log(k_{s-}h)_{A}\Vert^{2}ds\Pi(dh)=  t\left (\sigma^{2} + \int_{U_{0}} \Vert \log (h)_{A} \Vert^{2}\Pi(dh) \right )$ and thus we obtain  that almost surely
\[
\frac{M_{t}}{t} \underset{t \to +\infty}{\longrightarrow} 0.
\]
Moreover, making the change of variable $h'= k_{s-} h k_{s-}^{-1}$ we obtain using the  $K$-invariant by conjugation of $\Pi$ 
\begin{align*}
 \int_{0}^{t} \int_{U_{0}} & \left ( \log(k_{s-}h)_{A} -r(h)\sum_{i=1}^{d} \theta^{i}(h) \{ Ad(k_{s-})V_{i} \}_{\mathcal{A}} \right ) ds \Pi(dh)\\ & = \int_{0}^{t} \int_{U_{0}} \log (h')_{A} -r(h') \sum_{i=1}^{d} \theta^{i}(k_{s-}^{-1}h') \{Ad(k_{s-})V_{i}\}_{\mathcal{A}} \  \Pi(dh')   ds \\
 &= \int_{0}^{t} \int_{U_{0}} \log (h')_{A} -r(h') \sum_{i=1}^{d} \sum_{j=1}^{d} (k_{s-})_{ij} \theta^j(h') (k_{s-})_{i1}V_{1} \  \Pi(dh') ds\\
 &= \int_{0}^{t} \int_{U_{0}} \log (h')_{A} -r(h') \theta^{1}(h') V_{1} \  \Pi(dh') ds \\
 &=t  \int_{\mathbb{S}^{d-1}} \int_{0}^{1} \log (S(r,\theta))_{A} -r\theta^{1}V_1 \  \nu(dr) d\theta\\
\end{align*}

By \eqref{link2} we have $\log (S(r,\theta))_{A} = \log ( \cosh(r) + \theta^{1} \sinh(r) ) V_{1} $ and the previous term  equals 
\begin{align*}
t  \int_{\mathbb{S}^{d-1}} \int_{0}^{1} \log ( \cosh(r) + \theta^{1} \sinh(r) )  -r\theta^{1}\  \nu(dr) d\theta  V_1 &= t \int_0^1 \left (  \int_{-1}^{1} \log  \left ( \cosh(r) + u \sinh(r)  \right )  -r u \frac{du}{2} \right ) \nu(dr) V_1 \\
&= t \int_0^1 \frac{r \cosh(r) - \sinh(r) }{\sinh(r)} \nu(dr) 
\end{align*}
It remains to consider the asymptotic behavior of the stochastic integral $\int_{0}^{t}\int_{U_{0}^{c}} \log(k_{s-}h)_{A} N(ds,dh)$. We know that there exist $T_i$ jump times of a Poisson process of intensity measure $\Pi(U_{0}^{c})$ and random variable $h_n$ i.i.d of law  $\Pi \vert_{U_{0}^c}/\Pi(U_{0}^{c}) $ and independent of $(T_i)_i$ such that
\[
\int_{0}^{t}\int_{U_{0}^{c}} \log(k_{s-}h)_{A} N(ds,dh) = \sum_{n, T_{n }\leq t} \log(k_{T_{n}^{-}} h_{n})_{A}. 
\]
Moreover, by invariance of $\Pi$ under conjugation by $K$ we check easily that the random variables $h'_{n}:=Ad(k_{T_{n}^{-}}) h_{n}$ are i.i.d of common law  $\Pi \vert_{U_{0}^c}/\Pi(U_{0}^{c}) $  and we have
\[
\int_{0}^{t}\int_{U_{0}^{c}} \log(k_{s-}h)_{A} N(ds,dh) = \sum_{n=1}^{N_{t}} \log (h'_{n})_{A},
\]
where $N_{t}$ is a Poisson process of intensity measure $\Pi(U_{0}^{c})$ and independent of $(h'_n)_n$. Moreover since $\int_{1}^{+\infty} r\nu(dr) < +\infty$ then $\log(h'_{n})_{A}$ is integrable and the law of large number ensures that
\[
\frac{1}{t} \int_{0}^{t} \int_{U_{0}^{c}} \log (k_{s-}h)_{A} N(ds,dh) \underset{t\to+\infty}{\longrightarrow} \Pi(U_{0}^{c}) \mathbb{E}[  \log(h'_{n})_{A} ] .
\]
The proof is ended by checking that
\[
\mathbb{E}[  \log(h'_{n})_{A} ] = \frac{1}{\Pi(U_{0}^{c})}\int_{1}^{+\infty}\frac{r \cosh(r) -\sinh(r)}{\sinh(r)} \nu(dr) V_{1}.
\]

\end{proof}

\begin{rmq}\label{rmqinteg}
When $\int_1^{+\infty}  r \nu(dr)  = + \infty $ we obtain $\mathbb{E} [ \left \vert  \log (h'_n)_A \right  \vert  ] = + \infty$. Nevertheless

\begin{align*}
\mathbb{E} \left [  - \min \left  (\log (h'_n)_A, 0  \right ) \right ]& = \int_1^{+ \infty} \int_{-1}^1 - \min \left( \log (\cosh(r + \theta^1 \sinh(r) ), 0   \right ) \frac{d\theta^1}{2} \nu(dr) \\ 
&= \int_1^{+\infty} \int_{e^{-r}}^1 -\log (v) \frac{dv}{1 \sinh(r)} \nu(dr) \\
&= \int_1^{+\infty} \frac{1}{2 \sinh(r)} (1- e^{-r} (r+1) ) \nu(dr) < + \infty.
\end{align*}
Now, applying a generalized law of large numbers we deduce that almost surely 
\[
\frac{1}{t} \int_0^t  \int_{U_0^c}  \log (k_{s^{-}} h)_A N(ds, dh) \underset{t \to + \infty}{\longrightarrow} + \infty,
\]
and so 
\[
\frac{\alpha_t}{t}  \underset{t \to + \infty}{\longrightarrow} + \infty. 
\]
\end{rmq}
The following proposition establishes that $g_t$ is bounded in expectation on a finite time interval. This result is used to prove the convergence of $n_t$ in the next Proposition. 

\begin{prop}\label{supremum}
Fix $T>0$. Then
\[\E\left [ \sup_{t \in [0,T]}  r\left ( g_t \right) \right ] < + \infty. \]
\end{prop}
\begin{proof}
We cannot directly apply It\^o formula to $g \mapsto r(g)$ since it is not regular at $Id$. But we can find a smooth function $\tilde{r}$ such that $\tilde{r} \geq r$ on $U_0$ and $\tilde{r}=r$ on $U_0^c$. For such a function we have
\begin{align} \label{Ito2}
\tilde{r}(g_t)&= \tilde{r}(Id) + \hat{m}_t + \tilde{m}_t + I_t + J_t + \int_0^t \int_{U_0^c}\left ( \tilde{r}(g_{s^{-}}h)-\tilde{r}(g_{s^{-}} )  \right ) N(ds,dh). 
\end{align}
Where $\hat{m}_t:= \sigma \int_0^t V_i^l \tilde{r}(g_{s^{-}} )dB_s^i$ and $\tilde{m}_t:= \int_0^t \int_{U_0} \left ( \tilde{r}(g_{s^{-}}h)- \tilde{r}(g_{s^{-}})\right ) \tilde{N}(ds,dh)$ are martingales,  $I_t:=\sigma^{2}  \int_0^t \sum_{i=1}^{d} \left (  V_i^l \right )^2\tilde{r}(g_{s^{-}} ) ds $ and 
\[
J_t:=\int_0^t \int_{r \in [0,1]} \int_{\theta \in \mathbb{S}^{d-1}}  \left ( \tilde{r} \left (g_{s^{-}} S(r, \theta) \right)  -\tilde{r}(g_{s^{-}}) - r\sum_{i=1}^{d} \theta^i V_i^l \tilde{r}(g_{s^{-}})\right ) d \theta  \nu(dr)  ds
\] 
are processes with finite variation. \\

To prove the proposition we will bound the supremum on $[0,T]$ of each of these five terms by means of $\Vert X^l \tilde{r} \Vert_\infty$  and $\Vert (X^l)^2 \tilde{r} \Vert_\infty$ for some $X \in \mathrm{Vect} \{ V_i, i=1,\dots,d \}$. Thus we need the following lemma.

\begin{lem}
For $X \in \mathcal{P}=\mathrm{Vect} \{ V_i, i=1,\dots,d \}$ we have $\Vert X^l \tilde{r} \Vert_\infty < + \infty$ and $\Vert (X^l)^2 \tilde{r} \Vert_\infty < +\infty$.
\end{lem}

\begin{proof}[Proof of the lemma]
$U_0$ being a compact set it suffices to show that the supremum is finite on $U_0^c$. Since $\tilde{r}= r$ on $U_0^c$ it remains to prove that $\sup_{g\in U_0^c} \vert X^l r(g) \vert < +\infty$ and $ \sup_{g\in U_0^c} \vert (X^l)^2 r(g) \vert < +\infty$. The polar decomposition of $g$ can be written $g=\tilde{k} \exp(r(g) V_1) k$ for some $\tilde{k}, k \in K$. Setting $x \in \R^d$ such that $X= \sum_i x^i V_i $ we have
\[
r(g \exp(sX) ) = r\left (\exp(r(g) V_1) \exp \left (s \sum_{i=1}^d (kx)^i V_i \right ) \right ). 
\]
Let $\theta \in \mathbb{S}^{d-1}$ be such that $(kx)^i = \Vert x \Vert \theta^i$. Then we compute explicitly, for $g \in U_0^c$:
\begin{align*}
 r\left (\exp(r(g) V_1) \exp \left (s \sum_{i=1}^d (kx)^i V_i \right ) \right ) &= (\cosh)^{-1} \left (  \cosh(r(g)) \cosh(s \Vert x \Vert ) + \theta^1 \sinh(r(g)) \sinh(s \Vert x \Vert) \right ) \\
 &= r(g) + s \Vert x \Vert  \theta^1 + \frac{s^2}{2} \Vert x \Vert^2 \frac{\cosh(r(g))}{\sinh(r(g))} (1- (\theta^1)^2) + O(s^3).
\end{align*}
Then it comes that
\[
X^l r(g)=\left.  \frac{d}{ds} r(g \exp(sX) ) \right \vert_{s=0} = \Vert x \Vert \theta^1,
\]
and thus $\sup_{g \in U_0^c} \left  \vert X^l r(g) \right \vert \leq \Vert x \Vert$. \\
Moreover
\[
(X^l)^2 r(g) = \left. \frac{d^2}{ds^2} r(g \exp(sX) ) \right \vert_{s=0} =\Vert x \Vert^2 \frac{\cosh(r(g))}{\sinh(r(g))} (1- (\theta^1)^2)
\]
and so $\sup_{g \in U_0^c} \left \vert (X^l)^2 r(g)   \right \vert \leq 2 \Vert x \Vert^2.$

\end{proof}
Return to the proof of Proposition \ref{supremum}. By Doob's norm inequalities (see \cite{Kal}) we obtain
\begin{align*}
\E \left [\sup_{t \in [0,T]} \vert \hat{m}_t\vert^2 \right ]\leq 4\E \left [  (\hat{m}_T)^2 \right ] \leq 4\sigma^2 \E \left [ \int_0^T \sum_{i=1}^{d} \left \vert  V_i^l \tilde{r}(g_{s^{-}} )  \right \vert^2 ds \right  ] \leq 4\sigma^2 T d \max_i \Vert V_i^l \tilde{r} \Vert^2_\infty < +\infty.
\end{align*}
and 
\begin{align*}
\E \left [ \sup_{t \in [0,T]} \vert \tilde{m}_t\vert^2 \right  ] &\leq 4\E \left [  (\tilde{m}_T)^2   \right ] \leq 4 \E \left  [ \int_0^T \int_{r\in [0,1]} \int_{\theta \in \mathbb{S}^{d-1}} \left \vert \tilde{r}\left (g_{s^{-}}S(r, \theta) \right)-\tilde{r}(g_{s^{-}} )   \right \vert^2  \nu(dr) d\theta ds \right ] \\ &\leq 8 d T \int_0^1 r^2 \nu(dr)  \max_i \Vert V_i^l \tilde{r} \Vert^2_\infty < +\infty.
\end{align*}
 We have also
\begin{align*}
\E  \left [ \sup_{t \in [0,T]} \vert I_t \vert  \right ] \leq \sigma^2 T d \max_i \Vert (V_i^l)^2 \tilde{r} \Vert_\infty < +\infty .
\end{align*}
Moreover applying a Taylor inequality  to $u \in [0,1] \mapsto \tilde{r}(g_{s^{-}}S(ur,\theta) )$ it comes
\begin{align*}
\left \vert \tilde{r} \left (g_{s^{-}} S(r, \theta) \right) -\tilde{r}(g_{s^{-}}) - r\sum_{i=1}^{d} \theta^i V_i^l \tilde{r}(g_{s^{-}})    \right \vert \leq \frac{1}{2} r^2 \sup_{\theta \in \mathbb{S}^{d-1}} \left \Vert \left (\sum_i \theta^i V_i^l \right)^{2} \tilde{r} \right \Vert_\infty < +\infty
\end{align*}
and thus
\begin{align*}
\E  \left [ \sup_{t \in [0,T]} \vert J_t \vert  \right ] \leq \frac{T}{2} \left ( \int_0^1 r^2 \nu(dr) \right ) \sup_{\theta \in \mathbb{S}^{d-1}} \left \Vert \left (\sum_i \theta^i V_i^l \right)^{2}\tilde{r} \right \Vert_\infty < +\infty.
\end{align*}
Finally, to bound the supremum of the last term of \eqref{Ito2} we need the assumption \ref{integrability}. Indeed, since $\left \vert \tilde{r}(g_{s^{-}}S(r,\theta)) -\tilde{r}(g_{s^{-}})   \right \vert \leq r \sup_{\theta \in \mathbb{S}^{d-1}} \left \Vert \left (\sum_i \theta^i V_i^l \right)\tilde{r} \right \Vert_\infty$ we obtain
\begin{align*}
\E  \left [\sup_{t\in [0,T] }\left \vert \int_0^t \int_{U_0^c}  \tilde{r}(g_{s^{-}}h) -\tilde{r}(g_{s^{-}})   ds \Pi(dh) \right \vert  \right ]\leq T d \left (\int_{1}^{+\infty} r \nu(dr) \right )\max_i \left \Vert  V_i^l \tilde{r} \right \Vert_\infty < +\infty.
\end{align*}
\end{proof}

We can now state the main result of this section, the convergence of $n_t$ to some asymptotic random variable $n_\infty$ and the speed of convergence. 

\begin{prop}\label{convht}
Denote by $b_t =( b_t^{i} )_{i=2,\dots, d}$ the $\R^{d-1}$-valued process such that \[n_t =\exp \left (  \sum_{i=2}^{d} b_{t}^{i-1} (V_i-V_{1i} ) \right ). \] Then $b_t$ converges almost surely to $b_\infty$ exponentially fast with rate $\alpha$, i.e
\begin{align}
\limsup_{t \to +\infty}\frac{1}{t} \log \Vert b_t -b_{\infty} \Vert \leq - \alpha. \label{ht}
\end{align}
As a consequence $n_t$ converges almost-surely to $n_\infty := \exp \left ( \sum_{i=2}^{d} b_{\infty}^{i-1}(V_i-V_{1i} )   \right )$ and defining $h_t := e^{- t\alpha V_1} n_{\infty}^{-1} g_t$ we obtain
\begin{align}
\lim_{t \to +\infty} \frac{1}{t}  r\left ( h_t\right ) = 0 \quad \mathrm{a.s}.  \label{rform}
\end{align}
\end{prop}
\begin{proof}
Denoting by $[t]$ the integer part of $t$ we decompose $b_t = \sum_{j=1}^{[t]} ( b_j -b_{j-1} )  + b_t -b_{[t]} $. To prove that $b_t$ converge and \eqref{ht}  holds it is sufficient to verify that
\begin{align}
\limsup_{j \to \infty } \frac{1}{j} \log  \sup_{s\in [0,1]} \Vert b_j -b_{j +s} \Vert  \leq -\alpha \quad \mathrm{a.s} \label{dec}.
\end{align}
For $j \in \N$ and $s\in [0,1]$ we have $g_{j}^{-1} g_{j+s} =  k_{j}^{-1} a_{j}^{-1} n_{j}^{-1} n_{j+s} a_{j+s} k_{j+s}$ and thus
\begin{align} n_{j}^{-1} n_{j+s} = a_{j} \left (  g_{j}^{-1} g_{j+s} k_j \right )_N a_{j}^{-1}. \label{eq1} \end{align}
Denote by $\tilde{b}_{j,s}=  (\tilde{b}_{j,s}^{i})_{i=2, \dots, d} \in \R^{d-1}$ such that $\left (  g_{j}^{-1} g_{j+s} k_j \right )_N = \exp \left ( \sum_{i=2}^{d} \tilde{b}_{j,s}^{i} (V_i -V_{1i}) \right )$. Then, \eqref{eq1} implies that
\begin{align}
b_{j+s} -b_j = e^{- \alpha_j} \tilde{b}_{j,s}. \label{hjs}
\end{align}
Since $g_t$ is a Levy process, $ \left (\sup_{s\in [0,1] } r(g_{j}^{-1} g_{j+s}) \right )_{j\in \N}$ are i.i.d random variables. Moreover, by Proposition \ref{supremum} their common expectation $ \E [ \sup_{s \in [0,1]} r(g_{s})]$ is finite. Thus, applying the law of large number it comes
\begin{align}
\frac{1}{j} \sup_{s\in [0,1] } r(g_{j}^{-1} g_{j+s} ) \underset{j \to \infty}{\longrightarrow} 0 \quad \mathrm{a.s.} \label{LDGN}
\end{align}
Since $r( (g_{j}^{-1}g_{j+s} )_N ) \leq 2 r (g_{j}^{-1} g_{j+s})$ and, by  \eqref{link}, $ \Vert \tilde{b}_{j,s} \Vert^2 = 2 \left ( \cosh \left( r \left(  g_{j}^{-1} g_{j+s} k_{j} \right )_N \right ) -1 \right )$ we deduce from \eqref{LDGN}  that for $\eps >0$ there exists $j_0$ such that for $j>j_0$ 
\begin{align}
\Vert \tilde{b}_{j,s} \Vert \leq e^{\eps j}. \label{hjs2}
\end{align}
Since, by Proposition \ref{alpha}, $\alpha_j = j\alpha + o(j)$ thus \eqref{dec} follows from \eqref{hjs} and \eqref{hjs2}. To finish the proof of the proposition we need to check \eqref{rform}. We have
\[
r(h_t)=r( e^{-\alpha t V_1} n_{\infty}^{-1} n_t a_t ) \leq r( e^{-\alpha t V_1}a_t) + r( a_{t}^{-1} n_{\infty}^{-1}n_{t} a_{t}),
\]
and $r( e^{-\alpha t V_1}a_t) = \vert \alpha_t -\alpha t \vert  = o(t)$ and we obtain from \eqref{link} ( since $ a_{t}^{-1} n_{\infty}^{-1}n_{t} a_{t} \in N$), 
\[
r (a_{t}^{-1} n_{\infty}^{-1}n_{t} a_{t})= \cosh^{-1} \left ( 1 + e^{2 \alpha_t} \frac{\Vert b_t -b_\infty \Vert^2}{2}   \right ).
\]
So by \eqref{ht} we have also $r (a_{t}^{-1} n_{t}^{-1}n_{\infty} a_{t}) = o(t)$ and \eqref{rform} holds. 
\end{proof}
\begin{rmq}\label{rmq}
\begin{itemize}
\item Without integrability condition, we can nevertheless show that $g_t$ satisfy the \emph{irreducibility} and \emph{contraction} conditions of \cite{Liao04} and deduce that $\alpha_t$ converges almost surely to $+\infty$ and $n_t$ converges to $n_\infty$. Nevertheless, by remark  \ref{rmqinteg}, we obtain in the case where $\int_1^{+ \infty} r \nu(dr) = +\infty$, that almost surely $\frac{\alpha_t}{t}$ converges to $+\infty$.
\item In \cite{Liao04} the author uses a stronger hypothesis to prove the rate of convergence of a L\'evy process in a semi-simple group. It corresponds in our case to assume that $\int_0^{+\infty} r \nu(dr) < + \infty $. 
\end{itemize}
\end{rmq}

\subsection{Asymptotic behavior of the $\R^{1,d}$-component} 
The linear endomorphism $\xi \mapsto \exp ( V_1 ) \xi $ is diagonalisable with eigenvalues $-1, 0, +1$. Denote by $U^{-}$, $U^{0}$ and $U^{+}$ the respective eigenspaces. Explicitly $U^{-}= \mathrm{Vect} \{ e_0 - e_1  \} $, $U^{0}= \mathrm{Vect} \{  e_2, \dots, e_d  \}$ and $U^{+}= \mathrm{Vect} \{ e_0 + e_1  \} $. For $\xi \in \R^{1,d}$ we denote by $(\xi)^{-}$, $(\xi)^{0}$ and $(\xi)^{+}$ its projection on each eigenspace. Explicitly we obtain
\begin{align*}
(\xi )^{-}= -\frac{1}{2} q( \xi , e_0 + e_1 )& (e_0 -e_1), \quad  \quad(\xi)^{+}= \frac{1}{2} q(\xi, e_0 -e_1 ) (e_0 +e_1) \\
&(\xi)^{0}= \sum_{i=2}^{d} q(\xi, e_i ) e_i. 
\end{align*}

Recall that by definition, $\xi_t = \int_0^t g_s (e_0) ds$. The following proposition gives the asymptotic behavior of $\xi_t$. 
\begin{prop}\label{xit}
There exists an asymptotic random variable $\lambda_{\infty}>0$ such that 
\[
(n_\infty^{-1} \xi_t )^{-}  \underset{t \to + \infty}{\longrightarrow} \lambda_\infty (e_0 -e_1),
\] 
and moreover
\begin{align}
\limsup_{t \to + \infty} \frac{1}{t} \log \Vert (n_\infty^{-1} \xi_t )^{-} - \lambda_\infty (e_0 -e_1) \Vert \leq -\alpha. \label{convxit}
\end{align}
We also have
\begin{align}
\limsup_{t \to +\infty}  \frac{1}{t} \log \Vert (n_\infty^{-1} \xi_t )^{0}  \Vert \leq 0,  \quad \mathrm{and}, \quad  
\limsup_{t \to +\infty}  \frac{1}{t} \log \Vert (n_\infty^{-1} \xi_t )^{+}  \Vert \leq \alpha.  
\end{align}
\end{prop}
\begin{proof}
We have
\begin{align}
&- \frac{1}{2} q( n_\infty^{-1} \xi_t , e_0 + e_1 ) = \int_0^t -\frac{1}{2}  q(n_\infty^{-1} n_s a_s (e_0), e_0 + e_1 ) ds   \label{integral}
\end{align}
and the integrand can be written
\begin{align*}
-\frac{1}{2} q(n_\infty^{-1} n_s a_s (e_0), e_0 + e_1 ) &= -\frac{1}{2} q(e^{-\alpha s V_1}n_\infty^{-1} n_s a_s (e_0),e^{-\alpha s V_1}(e_0 + e_1) ) \\ &= -\frac{1}{2} e^{-\alpha s}q(h_s (e_0),(e_0 + e_1) ),
\end{align*}
where $h_s =e^{- \alpha s } n_{\infty}^{-1} g_s$ as defined in Proposition \ref{convht}.

Denote by $(\tilde{r}_s , \tilde{\theta}_s) \in \R^{+} \times \mathbb{S}^{d-1}$ the polar decomposition of  $h_s (e_0) \in \Hyp^{d}$. So $\tilde{r}= r(h_s)$ and 
\[
- \frac{1}{2}e^{-\alpha s}q(h_s (e_0),(e_0 + e_1) ) = \frac{1}{2}e^{-\alpha s} \left ( \cosh(\tilde{r}_s ) -\tilde{\theta}_s^{1} \sinh(\tilde{r}_s)   \right ) \in \frac{1}{2}[e^{ -(\alpha s  +\tilde{r}_s)} , e^{-(\alpha s-\tilde{r}_s)} ].
\]
Proposition \ref{convht} ensures that $\tilde{r}_s= o(s)$ a.s, so fixing $\eps >0$ arbitrary small we can find $s_0 >0$ such that for all $s >s_0$ the integrand of \eqref{integral} is positive and bounded by $e^{-(\alpha - \eps)s}$. This ensures the convergence of $(n_\infty^{-1} \xi_t )^{-}$ to $\lambda_\infty (e_0 -e_1)$ with $\lambda_{\infty} >0$. Moreover for $t>s_0$
\begin{align*}
\left \vert  -\frac{1}{2}q( n_\infty^{-1} \xi_t , e_0 + e_1 ) - \lambda_\infty   \right \vert \leq \int_t^{+ \infty} e^{-(\alpha - \eps) s} ds = \frac{1}{\alpha -\eps } e^{-(\alpha -\eps)t},
\end{align*}
which prove \eqref{convxit}.

Now 
\begin{align*}
(n_{\infty}^{-1} \xi_t )^0 = \sum_{i=2}^{d} \int_0^t q( n_{\infty}^{-1} n_s a_s (e_0),e_i )ds e_i, 
\end{align*}
and for $i=2, \dots, d$ we have  $  q( n_{\infty}^{-1} n_s a_s (e_0),e_i )= q( e^{-\alpha s V_1} n_{\infty}^{-1} n_s a_s (e_0) , e_i )  =  \tilde{\theta}_s^{i} \sinh(\tilde{r}_s)$ so $\vert q( n_{\infty}^{-1} n_s a_s (e_0),e_i ) \vert \leq e^{\tilde{r}_s}$ and  this ensures that $\limsup_{t \to +\infty}  \frac{1}{t} \log \Vert (n_\infty^{-1} \xi_t ){0}  \Vert \leq 0$.

Moreover, 
\begin{align*}
(n_{\infty}^{-1} \xi_t ){+} = \left ( \int_0^t \frac{1}{2} q\left ( n_{\infty}^{-1} n_s a_s (e_0), e_0 -e_1 \right  )ds \right ) (e_0 -e_1 )
\end{align*}
and 
\begin{align*}
\frac{1}{2} q( n_{\infty}^{-1} n_s a_s (e_0), e_0 -e_1 )&= \frac{1}{2} e^{\alpha s} q(h_t (e_0), e_0 -e_1 ) \\
&= \frac{1}{2}  e^{\alpha s} \left (  \cosh(\tilde{r}_s ) - \tilde{\theta}_s^{1} \sinh(\tilde{r}_s)  \right ) \in \frac{1}{2} [ e^{\alpha s - \tilde{r}_s } , e^{\alpha s + \tilde{r}_s } ].
\end{align*}
So 
\[
\Vert (n_{\infty}^{-1} \xi_t )^{+} \Vert \leq \Vert (n_{\infty}^{-1} \xi_{s_0} )^{+} \Vert + \int_{s_0}^{t} e^{(\alpha + \eps) s} ds,
\]
and thus $\limsup_{t \to +\infty}  \frac{1}{t} \log \Vert (n_\infty^{-1} \xi_t )^{+}  \Vert \leq \alpha.$
\end{proof}  
\subsection{Geometric description of the convergence}

Denote by $\mathbf{p}: \R^{1,d} \setminus \{  0 \} \to \Prob^d \R$ the projection onto the projective space of dimension $d$. The hyperbolo\"id $\Hyp^d$ is mapped onto the interior of a projective ball and its boundary  ($\partial \Hyp^d \simeq \mathbb{S}^{d-1}$) is the image of the $q$-isotropy cone
\[
\partial \Hyp^d := \mathbf{p} \left ( \{ \xi , q(\xi) =0 \} \setminus \{ 0 \} \right ).
\]

From the relation
\begin{align*}
q( \dot{\xi}_t, n_t (e_0 +e_1) ) = q( g_t (e_0) , n_t (e_0 +e_1) ) = q( e_0 , a_t^{-1} (e_0 +e_1) ) = e^{-\alpha_t} \underset{t \to \infty}{\longrightarrow} 0,
\end{align*}
we deduce that all limit points of $\mathbf{p}(\dot{\xi}_t)$ are $q$-orthogonal to $\theta_\infty:= \mathbf{p}( n_{\infty} (e_0 +e_1) )$. Since the only point of $\widebar{\mathbf{p}(\Hyp^d)}$
 which is $q$-orthogonal to $\theta_\infty$ is $\theta_\infty$ itself it comes that   $ \mathbf{p}(\dot{\xi}_t)$ converges to $\theta_{\infty}$ in $\Prob^d \R$. Now, identifying  $\Prob^d \R$ with its affine chart $\{ \xi, \ \xi^0 =1 \}$ we can consider that $\theta_\infty \in \mathbb{S}^d$. From \eqref{link} we deduce that $r_t \to +\infty$ and since $\mathbf{p}(\dot{\xi}_t) = \mathbf{p}( e_0 + \theta_t \frac{\sinh(r_t)}{\cosh(r_t)})$ it comes that $\theta_t$ converges to $\theta_{\infty}$ in $\mathbb{S}^d$. The two asymptotic random variables $\theta_\infty$ and $n_\infty$ are linked by 
 \[
 \mathbf{p}( e_0 + \theta_\infty ) = \mathbf{p}( n_\infty (e_0 + e_1) )
 \]
or more explicitly, $b_\infty \in \R^{d-1}$ (defined by $n_\infty = \exp \left ( \sum_{i =2}^{d}  b_\infty^{i-1} (V_1 -V_{1i} )\right )$ )  is the stereographic projection of $\theta_\infty$ 
\[
\theta_\infty = \frac{1}{1+ \Vert b_\infty \Vert^2} \left ( \begin{matrix} 1- \Vert b_\infty \Vert^2  \\ 2b_\infty  \end{matrix}   \right ).
\]

Concerning the asymptotic behavior of $\xi_t$, Proposition \ref{xit} ensures that $ q( \xi_t, n_{\infty} (e_0 + e_1 ) )$ converges to $\lambda_\infty$. Thus geometrically $\xi_t$ is asymptotic to an affine hyperplan which is $q$- orthogonal to $ n_{\infty} (e_0 + e_1 )$ (or $e_0 + \theta_\infty$ ) and passing by $\lambda_\infty (e_0 -e_1)$ .

\section{Lyapunov spectrum and stable manifolds} \label{LyapSpec}
\subsection{Lyapunov spectrum} 
The Levy process $\tilde{g}_t$, with values in $\widetilde{G}$ and starting at some $\tilde{g}$, can be obtained by solving the following left invariant stochastic integro-differential equation in $\widetilde{G}$

\begin{align}
 \forall f \in C^2( \tilde{G}), \quad f(\tilde{g}_t)& = f(\tilde{g}) + \sigma  \sum_{i=1}^d\int_0^t V_i^{l}f(\tilde{g}_{s^{-}})\circ dB^{i}_s + \int_0^t H_0^{l} ( \tilde{g}_{s^{-}})ds +\int_0^t \int_{U_0}  \left ( f(\tilde{g}_{s^{-}}h) -f(\tilde{g}_{s^{-}}) \right ) \tilde{N}(ds,dh)  \notag \\ 
&+ \int_0^t  \int_{U_0} \left ( f(\tilde{g}_{s^{-}}h) -f(\tilde{g}_{s^{-}}) -r(h) \sum_{i=1}^{d}  \theta^{i}(h) V_i^l f(\tilde{g}_{s^{-}}) \right )  ds \Pi(dh) \label{SDEtilde} \\ 
& + \int_0^t \int_{(U_0)^c} \left ( f(\tilde{g}_{s^{-}}h) -f(\tilde{g}_{s^{-}}) \right ) N(ds,dh). \notag
\end{align}

This stochastic differential equation induces a stochastic flow $\varphi_t$ in $\widetilde{G}$  which maps $\tilde{g}$ on the solution at time $t$ and starting at $\tilde{g}$ of \eqref{SDEtilde}. By left invariance $\varphi_t$ is also defined by 
\[
\begin{matrix}
\varphi_t : &  \widetilde{G} & \longrightarrow &  \widetilde{G} \\
& \tilde{g} & \longmapsto & \tilde{g} \tilde{g}_t,
\end{matrix}
\]
where $\tilde{g}_t$ is starting at $\mathrm{Id}$.

Denote by $\Vert \cdot \Vert $ any norm on $\mathrm{Lie}( \widetilde{G} )$ and  by $\Vert \cdot \Vert_{\tilde{g}}$ the left invariant (Finsler) metric associated in $ \widetilde{G}$ on $T_{\tilde{g}} \widetilde{G}$. For $v \in T_{\tilde{g}}  \widetilde{G}$ we aim to investigate the asymptotic exponential rate of growth or decay of $\Vert d \varphi_t (\tilde{g})(v) \Vert_{\varphi_t (\tilde{g})}$. Denote by $L_{\tilde{g}}$ the left translation  by  $\tilde{g}$ in $ \widetilde{G}$. By left invariance of the flow, $\Vert d \varphi_t (\tilde{g})(v) \Vert_{\varphi_t(\tilde{g})}= \Vert d \varphi_t(\mathrm{Id})( \widetilde{X}) \Vert_{\varphi_t(\mathrm{Id}) }$ where $ \widetilde{X}:= (dL_{\tilde{g}})^{-1}(v) \in T_{\mathrm{Id}}  \widetilde{G}=\mathrm{Lie}( \widetilde{G})$. For $\tilde{g}= (g, \xi) \in  \widetilde{G}$ and $ \widetilde{X}, \widetilde{Y}\in  \mathrm{Lie} ( \widetilde{G} )$ it comes
\begin{align}
\mathrm{Ad}(\tilde{g}) ( \widetilde{X}) &= \left ( \mathrm{Ad}(g)(X), gx -  \mathrm{Ad}(g)(X)\xi  \right)   \label{Ad} \\ 
\mathrm{ad}(\widetilde{Y})(\widetilde{X}) &= \left ( \mathrm{ad}(Y)(X), Yx -Xy\right ).
\end{align}
The endomorphism $\widetilde{X} \mapsto \mathrm{ad}(V_1,0)( \widetilde{X})$ is diagonalisable on $\mathrm{Lie}(\widetilde{G})$. Its eigenvalues are $-1,0 , 1$ and we denote by $ \widetilde{U}^{-}$, $ \widetilde{U}^{0}$ and $ \widetilde{U}^{+}$ the eigenspaces associated.

 We can check that 
\begin{align*}
\widetilde{X} \in \widetilde{U}^{+} &\Longleftrightarrow X \in \widebar{\mathcal{N}} \ \mathrm{and} \ x\in U^{+} \\
\widetilde{X} \in \widetilde{U}^{0} &\Longleftrightarrow X \in \mathcal{A}\oplus \mathcal{M}  \ \mathrm{and} \ x\in U^{0} \\
\widetilde{X} \in \widetilde{U}^{-} &\Longleftrightarrow X \in \mathcal{N} \ \mathrm{and} \ x \in U^{-}.
\end{align*}

Set \[ \tilde{g}_\infty := \left (n_\infty , \lambda_\infty (e_0 -e_1) \right) \in \widetilde{G} \] and $V_\infty^{-}:= \mathrm{Ad}(\tilde{g}_\infty)\left (  \widetilde{U}^{+} \right ) $, $V_\infty^{0}:= \mathrm{Ad}(\tilde{g}_\infty) \left ( \widetilde{U}^{0} +  \widetilde{U}^{+}\right )$.

We denote by \[ \tilde{h}_t := \left ( e^{-t \alpha V_1}, \ 0 \right ) \tilde{g}_{\infty}^{-1} \tilde{g}_t,\] where we recall  that $\tilde{g}_t:= \varphi_t(\mathrm{Id})$ is starting at $\mathrm{Id}$.

\begin{thm}\label{Lyap}
Let $ \widetilde{X} \in \mathrm{Lie}(  \widetilde{G} )$. For almost every trajectory
\[
\frac{1}{t} \log \Vert  d \varphi_t (\mathrm{Id})( \widetilde{X})  \Vert_{\varphi_t(\mathrm{Id})} \underset{t \to +\infty}{\longrightarrow} \left \{  \begin{matrix} \alpha & \mathrm{if} &  \widetilde{X} \in  \mathrm{Lie}(  \widetilde{G} ) \setminus V_\infty^0 \\ 0 & \mathrm{if} &  \widetilde{X} \in  V_\infty^0 \setminus V_\infty^-  \\ -\alpha & \mathrm{if} &  \widetilde{X} \in  V_\infty^- \setminus \{ 0 \}  \end{matrix}\right.
\]
\end{thm}
\begin{proof}
By left invariance of $\Vert \cdot \Vert$, $\Vert  d \varphi_t (\mathrm{Id})( \widetilde{X})  \Vert_{\varphi_t(\mathrm{Id})} = \Vert \mathrm{Ad}(\tilde{g_t}^{-1}) († \widetilde{X} )\Vert$. We set, for $\tilde{g} \in  \widetilde{G}$
\[
\Vert \mathrm{Ad}(\tilde{g}) \Vert := \sup_{ \widetilde{X} \neq 0} \frac{\Vert \mathrm{Ad}(\tilde{g})( \widetilde{X})\Vert}{\Vert  \widetilde{X} \Vert}. 
\]
Let $ \widetilde{X} \in \mathrm{Lie}(  \widetilde{G})$. Writting $\tilde{g}_t^{-1}= \left (  \tilde{h}_t \right )^{-1}  \left ( e^{-t \alpha V_1} ,0 \right ) \tilde{g}_{\infty}^{-1}$ we deduce that
\begin{align}
\frac{\Vert \mathrm{Ad} \left (   \left ( e^{-t \alpha V_1} ,0 \right ) \tilde{g}_{\infty}^{-1} \right )( \widetilde{X}) \Vert }{\Vert \mathrm{Ad} (  \tilde{h}_t  )  \Vert}  \leq \Vert &\mathrm{Ad}(\tilde{g}_t^{-1})( \widetilde{X})\Vert \notag \\ 
&\leq \Vert \mathrm{Ad} (\tilde{h}_t)^{-1}  \Vert \Vert \mathrm{Ad} \left (   \left ( e^{-t \alpha V_1} ,0 \right ) \tilde{g}_{\infty}^{-1} \right )( \widetilde{X}) \Vert. \label{ineg} 
\end{align}
Suppose for the moment that 
\begin{align}
&\limsup_{t \to + \infty} \frac{1}{t} \log \left \Vert \mathrm{Ad}  (  \tilde{h}_t )^{-1}  \right \Vert \leq 0  \label{limsup1} \\ 
\mathrm{and} \quad &\limsup_{t \to + \infty} \frac{1}{t} \log \Vert \mathrm{Ad}  ( \tilde{h}_t )  \Vert \leq 0. \label{limsup2}
\end{align}
Then we deduce from \eqref{ineg} that  $\frac{1}{t} \log \Vert \mathrm{Ad}(\tilde{g}_t^{-1})( \widetilde{X})\Vert $ and $\frac{1}{t} \log \Vert \mathrm{Ad} \left (   \left ( e^{-t \alpha V_1} ,0 \right ) \tilde{g}_{\infty}^{-1} \right )( \widetilde{X}) \Vert $ have the same limit when $t$ goes to $\infty$. The linear isomorphism $\mathrm{Ad} \left ( e^{-t \alpha V_1} ,0 \right )$ is diagonalisable with eigenvalues $e^{-\alpha t}$, $1$ and $e^{\alpha t}$ associated respectively to the eigenspaces $ \widetilde{U}^{+}$, $ \widetilde{U}^{0}$ and $ \widetilde{U}^{-}$. Decomposing $\mathrm{Ad}(\tilde{g}_\infty)^{-1} ( \widetilde{X})$ in the direct sum $ \widetilde{U}^{-} \oplus  \widetilde{U}^{0} \oplus  \widetilde{U}^{+}$ and using a Euclidean norm $\Vert \cdot \Vert$ on $\mathrm{Lie}({ \widetilde{G}})$ for which this decomposition is orthogonal, we deduce easily the theorem (note that the convergence is independant of the chosen norm).

Thus it remains to prove \eqref{limsup1} and \eqref{limsup2}. We have
\begin{align*}
\tilde{h}_t=\left ( e^{-t \alpha V_1}, \ 0 \right ) \tilde{g}_{\infty}^{-1}\tilde{g}_t &= \left ( e^{-t \alpha V_1}, \  0 \right ) \left (n_\infty^{-1} , \  -\lambda_{\infty} (e_0 -e_1)†\right  ) \left (n_t a_t k_t , \ \xi_t \right ) \\ &= \left ( h_t, \  \ e^{-t \alpha V_1 } \left ( n_{\infty}^{-1}\xi_t \right ) - \lambda_\infty e^{t \alpha} (e_0 -e_1)   \right ) \\ &= \left (\mathrm{Id} , \ e^{-t \alpha V_1 } \left (n_{\infty}^{-1}\xi_t \right ) - \lambda_\infty e^{t \alpha} (e_0 -e_1) \right  ) \ \left ( h_t , \ 0 \right )
\end{align*}

Let $\eps >0$. By Proposition \ref{convht} we can find $t_0 >0$ such that $\forall t >t_0$ $r(h_t) \leq \eps t$ and by Proposition \ref{xit}  we have 
\begin{align*}
\Vert e^{-t \alpha V_1 } \left ( n_{\infty}^{-1}\xi_t \right )- \lambda_\infty e^{t \alpha} (e_0 -e_1)   \Vert  & \leq e^{\alpha t} \Vert  (n_\infty^{-1}\xi_t )_{-} -\lambda_{\infty} (e_0-e_1) \Vert + \Vert  (n_{\infty}^{-1} \xi_t )_{0}\Vert + e^{-\alpha t}\Vert  (n_{\infty}^{-1}\xi_t)_{+} \Vert \\
& \leq e^{\eps t }
\end{align*}
Now using the following Lemma \ref{norms} we deduce easily \eqref{limsup1} and \eqref{limsup2}. 
\end{proof}
\begin{lem}\label{norms}
There exist positive constants $\alpha, \beta, \gamma$ such that for $g \in G$ and $\xi \in \R^{1,d}$ 
\begin{align*}
\Vert \mathrm{Ad} (g, 0)\Vert \leq \alpha e^{r(g)} \\
\Vert  \mathrm{Ad} (\mathrm{Id}, \xi )  \Vert \leq \beta \Vert \xi \Vert + \gamma.
\end{align*}
\end{lem}
\begin{proof}
All norms are equivalent and it suffices to check the inequalities for some particuliar norms. Let choose the following $SO(d)$-invariant euclidean norm on $\mathrm{Lie}( \widetilde{G})$
\[
\Vert (X,x) \Vert:= \sqrt{ \mathrm{Tr}(X^{t}X) + x^{t}x }.
\]
We obtain easily 
\[
\Vert   \mathrm{Ad}(g, \ 0)\Vert = e^{r(g)}.
\]
Taking now $\Vert (X,x) \Vert:= \sqrt{\mathrm{Tr}(X^{t}X)} + \sqrt{x^{t}x} $ we get 
\begin{align*}
\Vert \mathrm{Ad}(\mathrm{Id}, \ \xi ) (X,x) \Vert &= \sqrt{\mathrm{Tr}(X^{t}X)} + \Vert x - X\xi \Vert 
\leq \Vert (X, x )\Vert  + \Vert X\xi \Vert \\ &\leq  \Vert (X, x )\Vert + \sqrt{\mathrm{Tr}(X^{t}X)} \max_i \vert \xi^{i} \vert  \leq \Vert (X, x )\Vert \left  (1 +\max_i \vert \xi^{i} \vert \right ).
\end{align*}
Thus $\Vert \mathrm{Ad}(\mathrm{Id}, \ \xi ) \Vert \leq 1 + \alpha \Vert \xi \Vert$ for a constant $\alpha >0$ independant of $\xi$.

\end{proof}

\subsection{Stable manifolds}

First, remark that $V_{\infty}^{-}$ and $V_{\infty}^{0}$ are Lie sub-algebras of $\mathrm{Lie}( \widetilde{G} )$. Denote by 
\[ \mathcal{V}^{-}_{\infty}:= \exp (V_{\infty}^{-} ), \quad  \mathrm{and}  \quad \mathcal{V}^{0}_{\infty}:= \exp (V_{\infty}^{0}) \]
the closed subgroup of $ \widetilde{G}$ associated.

Fix now a euclidean norm $\Vert \cdot \Vert$ on $\mathrm{Lie}( \widetilde{G})$ which is $\mathrm{Ad}(K)$-invariant. Such a norm is of the form
\[
 \left \Vert \left (  \left ( \begin{matrix} 0 & b^{t} \\ b & C \end{matrix} \right ), \ x  \right ) \right \Vert := \sqrt{ \kappa^2 b^{t}b + \beta^2 \mathrm{Tr} \left ( C^{t}C \right )+ \gamma^2 \vec{x}^{t}\vec{x} + \delta^2 (x^0)^2 },
\] 
for some positive constants $\kappa$, $\beta$, $\gamma$ and $\delta$. We denote by $d$ the distance in $\widetilde{G}$ associated to the left invariant Riemanian metric induced by $\Vert \cdot \Vert$. To simplify notations, we denote by $d(g,h)$ the distance between $(g,0)$ and $(h,0)$ for $g, h \in G$.  

The following result shows that the stable manifold associated to $\varphi_t$ is $\varphi_{0} \mathcal{V}_{\infty}^{-}$. 

\begin{thm} \label{stable}
Let $\tilde{g}$ and $\tilde{g}'$ two distinct points in $ \widetilde{G}$. 
\begin{itemize}
\item If $\tilde{g}' \in \tilde{g} \mathcal{V}_{\infty}^{-}$ then 
\[
\frac{1}{t} \log d \left (  \varphi_t (\tilde{g}) , \varphi_t (\tilde{g}') \right ) \underset{ t \to + \infty }{\longrightarrow} -\alpha.
\]
\item If $\tilde{g}' \notin \tilde{g} \mathcal{V}_{\infty}^{-}$ then
\[
\liminf_{t \to \infty} d\left (  \varphi_t (\tilde{g}) , \varphi_t (\tilde{g}') \right ) >0.
\]
\end{itemize}
\end{thm}

The properties of $d$ we need in the proof  of Theorem \ref{stable} are sum up in the following proposition
\begin{prop}\label{dprop}
\begin{enumerate}[i)]
\item Left invariance \[\forall \tilde{g}, \tilde{h} \in  \widetilde{G},  \quad d( \tilde{g} , \tilde{g} \tilde{h} ) = d( \mathrm{Id} , \tilde{h} ). \]
Thus $ d( \mathrm{Id}, \ \tilde{g}^{-1} )= d( \mathrm{Id}, \ \tilde{g} )$ and triangularity inequality writes:
\[
\forall \tilde{g}, \tilde{h} \in  \widetilde{G},  \quad  d( \mathrm{Id} , \tilde{g} \tilde{h} ) \leq d( \mathrm{Id} , \tilde{g}  ) +d( \mathrm{Id}, \   \tilde{h} ) 
\]
\item  K-right invariance \[ \forall \tilde{g} \in  \widetilde{G} \ \mathrm{and} \ k \in K, \quad  d \left ( (k, 0) , \tilde{g} (k, 0) \right ) = d( \mathrm{Id} , \tilde{g} ) \]
\item For $ \widetilde{X} \in \mathrm{Lie}( \widetilde{G} ) $ 
\[
d \left ( \mathrm{Id} , \exp(  \widetilde{X} ) \right ) \leq \Vert  \widetilde{X} \Vert. 
\]
\item There exists a neighborood $\mathcal{O}$ of $0$ in $\mathrm{Lie}( \widetilde{G})$ and a constant $C>0$ such that
\[
\forall  \widetilde{X}  \in \mathcal{O}, \quad C \Vert  \widetilde{X} \Vert \leq d\left (\mathrm{Id}, \exp( \widetilde{X}) \right )
\]
\item $\forall (g, \xi ) \in  \widetilde{G}$
\[
d\left ( \mathrm{Id}, g \right ) \leq d \left ( \mathrm{Id},\ (g, \xi )   \right )
\]
\item  For $g=S(r, \theta) R$ and $g'= S(r', \theta')R' $ we have
\[
\frac{\kappa}{\sqrt{\kappa^2 +2\beta^2}}d\left ( S(r, \theta) , S(r', \theta') \right ) \leq d(g, g') \]
\item  For all $r \geq 0$ and $\theta \in \mathbb{S}^{d-1}$ \[d \left ( \mathrm{Id}, S(r,\theta ) \right )= \kappa r\] 
\end{enumerate}
\end{prop}

\begin{proof}[Proof of Proposition \ref{dprop}]
$i)$ and $ii)$.The left and $K$-right invariance follows from the definition of the metric as being a left invariant Riemannian metric on $\widetilde{G}$ defined from an $\mathrm{Ad}(K)$-invariant inner product on $\mathrm{Lie}(\widetilde{G})$. Inequality $iii)$ is obtained remarking that the length of the path $t \in [0,1] \mapsto \exp \left ( t \widetilde{X}  \right )$ is equal to $\Vert \widetilde{X} \Vert$. 

$iv)$. Denote by $\widehat{\exp}: \mathrm{Lie}(\widetilde{G})\to \widetilde{G}$ be the exponential map at $\mathrm{Id}$ induced by the metric $\Vert \cdot \Vert$ in $ \widetilde{G}$:  for $\tilde{X}\in \mathrm{Lie}(\widetilde{G})$, $\widehat{\exp}(\tilde{X})= \gamma_{\tilde{X}}(1)$ where $t \in [0,1] \mapsto \gamma_{\tilde{X}}(t)$  is the geodesic starting from $\mathrm{Id}$ in the direction $\tilde{X}$. The differential at $0$ of $\widehat{\exp}$ is known to be identity and there exists a sufficient small neighborhood $\mathcal{O}'$ of $0\in \mathrm{Lie}(\widetilde{G})$ such that: \begin{equation} \forall \tilde{X} \in \mathcal{O}', \quad \Vert \tilde{X} \Vert = d\left (I, \widehat{\exp}(\tilde{X}) \right). \label{truc} \tag{$*$}\end{equation} Furthermore, the map $\widehat{\exp}^{-1}\circ \exp $ can be defined in a neighborhood of $0$ and its differential at $0$ is the identity: $\widehat{\exp}^{-1}\circ \exp(\tilde{X})= \tilde{X}+o(\Vert \tilde{X} \Vert)$. So we can find $\mathcal{O}$ neighborhood of $0$ and $C>0$ such that for all $\tilde{X}\in \mathcal{O}$, $C \Vert \tilde{X} \Vert  \leq \Vert \widehat{\exp}^{-1}\circ \exp(\tilde{X}) \Vert \leq \frac{1}{C} \Vert \tilde{X} \Vert$. Taking $\mathcal{O}$ small enough so that $\widehat{\exp}^{-1}\circ \exp(\mathcal{O}) \subset \mathcal{O}'$, we can apply \eqref{truc} to $\widehat{\exp}^{-1}\circ \exp(\tilde{X})$, thus yielding $\Vert \widehat{\exp}^{-1}\circ \exp(\tilde{X}) \Vert= d\left(I, \exp(\tilde{X})\right)$ for every $\tilde{X}\in \mathcal{O}$.

$v)$. Each path $s \in [0,1]  \mapsto (g_s, \xi_s)$ joining $\mathrm{Id}$ to $(g,\xi)$ is of length $\int_{0}^1 \Vert (g_s^{-1} \dot{g}_s , g_s^{-1} \xi_s ) \Vert ds $ which is greater than $\int_{0}^1 \Vert (g_s^{-1} \dot{g}_s,  0 ) \Vert ds $ corresponding to the path $ s \in [0,1 ] \mapsto (g_s, 0)$ joining $\mathrm{Id}$ to $(g, 0)$. 

$vi)$. Consider a path $ s\in [0,1] \mapsto S(r_s, \theta_s)R_s$ joining $g$ to $ g'$. We compute, using dot notation for $\frac{d}{ds}$ 
\[
R_s^{-1} S(r_s, \theta_s)^{-1} \frac{d}{ds} (S(r_s, \theta_s) R_s )= \left ( \begin{matrix} 0 & \dot{r}_s \theta_s^t + \sinh(r_s) \dot{\theta_s}^t \\ \dot{r}_s \theta_s+ \sinh(r_s) \dot{\theta}_s & (\cosh(r_s) -1) \left ( \dot{\theta}_s \theta_s^{t} -\theta_s \dot{\theta}_s^t \right ) + R_s^{-1}\dot{R}_s 
  \end{matrix}  \right ).
\]
 Its length  $ l:= \int_0^1 \Vert R_t^{-1} S(r_t, \theta_t)^{-1} \frac{d}{dt} S(r_t, \theta_t) R_t  \Vert dt $ is larger than
 \[ \int_0^{1} \kappa \Vert  \dot{r}_s \theta_s+ \sinh(r_s) \dot{\theta}_s \Vert ds = \int_0^{1} \kappa  \sqrt{ (\dot{r}_s)^2 + \sinh(r_s)^2 \Vert \dot{\theta}_s  \Vert^2  }ds. \] Moreover, the path $s \mapsto S(r_s, \theta_s) $ which join $S(r, \theta)$ to $S(r', \theta')$ is of length 
 \begin{align}
 \int_0^1 &\sqrt{\kappa^2 \left ((\dot{r}_s)^2 + \sinh(r_s)^2 \Vert \dot{\theta}_s  \Vert^2  \right )+ \beta^2 \left ( 2 (\cosh(r) -1)^2 \Vert \dot{\theta}_s \Vert^2 \right ) }ds \\ & \quad \quad  \quad \quad \quad \quad \quad \quad\quad \quad\quad \quad\quad \quad\quad \quad\leq \frac{\sqrt{\kappa^2 +2 \beta^2}}{\kappa}  \int_0^{1} \kappa  \sqrt{ (\dot{r}_s)^2 + \sinh(r_s)^2 \Vert \dot{\theta}_s  \Vert^2  }ds.
 \end{align}
 Thus
 \[
 d( S(r, \theta), S(r', \theta') ) \leq   \frac{\sqrt{\kappa^2 +2 \beta^2}}{\kappa} l
 \]
and taking the infimum over all the path joining $g$ to $g'$ we obtain $vi)$. 
 \end{proof}

\begin{proof}[Proof of Theorem \ref{stable}]
Let  $ \widetilde{Y} \in \mathrm{Lie}( \widetilde{G} ) \setminus \{  0 \}$ be such that $\exp(  \widetilde{Y} ) =\tilde{g}^{-1} \tilde{g}' $. Then
\[
d \left (  \varphi_t (\tilde{g}) , \varphi_t (\tilde{g}') \right ) = d\left ( \mathrm{Id}, \ \tilde{g}_t^{-1} \exp\left (  \widetilde{Y} \right ) \tilde{g}_t \right )= d \left ( \mathrm{Id}, \ \exp \left ( \mathrm{Ad}(\tilde{g}_t^{-1})  (  \widetilde{Y}  )  \right )   \right ). 
\]
By Theorem \ref{Lyap}, if $ \widetilde{Y} \in \mathrm{Ad}(g_{\infty})( \widetilde{U}^{+})$ (i.e. if $g' \in g \mathcal{V}_{\infty}^{-}$) then $\Vert \mathrm{Ad}(g_t^{-1})  (  \widetilde{Y}  )   \Vert$ converge to $0$ exponentially fast with rate $\alpha$ and so for large $t$ it evolves in $\mathcal{O}$. Thus, using $iii)$ and $iv)$ of Proposition \ref{dprop} we obtain the first point of the theorem as a direct consequence of Theorem \ref{Lyap}.

Set $\widetilde{X}:= \mathrm{Ad}(\tilde{g}_{\infty}^{-1}) (  \widetilde{Y} )$, thus 
\[ \widetilde{Y} = \mathrm{Ad}(\tilde{g}_{\infty})\left ( \left ( \widetilde{X} \right )^{+} + \left ( \widetilde{X} \right )^{0} + \left ( \widetilde{X} \right )^{-} \right ),\]
and write
 \begin{align*} 
\left ( \widetilde{X} \right )^{+} = ( X^{+}, \ x^{+} ), &\ \mathrm{where} \ \  X^{+} \in \widebar{\mathcal{N}} \ \mathrm{and}  \ \  x^{+} \in U^{+} \\
\left ( \widetilde{X} \right )^{0} = ( X^{0}, \ x^{0} ), &\ \mathrm{where} \ \  X^{0} \in \mathcal{A} \oplus \mathcal{M}   \ \mathrm{and}  \ \  x^{0} \in U^{0} \\
\left ( \widetilde{X} \right )^{-} = ( X^{-}, \ x^{-} ), &\ \mathrm{where} \ \  X^{-} \in \mathcal{N} \ \mathrm{and}  \ \  x^{-} \in U^{-}.
 \end{align*}

Now suppose that $\tilde{g}' \notin \tilde{g} \mathcal{V}^{-}_{\infty}$ which is equivalent to  $\left ( \widetilde{X} \right )^{0} \neq 0$ or $\left ( \widetilde{X}\right )^{-} \neq 0$.

Suppose first that $\left ( \widetilde{X}\right )^{-} \neq 0$. Thus $\widetilde{Y} \in \mathrm{Lie}(\widetilde{G}) \setminus V_\infty^0$ and by Theorem \ref{Lyap} $\Vert  \mathrm{Ad}(\tilde{g}_t^{-1}) \widetilde{Y} \Vert $ converges to $+\infty$ exponentially fast. Now suppose by contradiction that ${ \displaystyle   \liminf_{t \to + \infty }d \left ( \mathrm{Id},\exp \left ( \mathrm{Ad}(\tilde{g}_t^{-1} )   \widetilde{Y}    \right ) \right ) =0   }$. Then we can find a $s_t$ such that $\displaystyle d \left ( \mathrm{Id}, \ \exp \left ( \mathrm{Ad}(\tilde{g}_{s_t}^{-1})(  \widetilde{Y}  ) \right )  \right )$ converges to $0$ and for large $t$ $ \mathrm{Ad}(\tilde{g}_{s_t}^{-1})(  \widetilde{Y}  )$ lies in $\mathcal{O}$. The inequality $iv)$ of Proposition \ref{dprop} give us the contradiction and we have proved the second point of the Theorem if $\left ( \widetilde{X}\right )^{-} \neq 0$.

So we can suppose $\left ( \widetilde{X}\right )^{-} = 0$ and $\left ( \widetilde{X}\right )^{0} \neq 0$.

\noindent {\bf First case: } $\mathbf{X^{0}†\neq 0}$ . By $v)$ of Proposition \ref{dprop}, $ {\displaystyle d( \mathrm{Id} , g_t^{-1} \exp( Y ) g_t ) \leq d\left ( \mathrm{Id}, \ \tilde{g}_t^{-1} \exp\left (  \widetilde{Y} \right ) \tilde{g}_t \right ) }$ and it remains to prove that ${\displaystyle \liminf_{t \to \infty} } d( \mathrm{Id} , g_t^{-1} \exp( Y ) g_t ) $ is positive.  But $ Y = \mathrm{Ad}(n_{\infty})(X)$ and $X= X^{+} + X^{0} \in \widebar{\mathcal{N}} \oplus \mathcal{A} \oplus \mathcal{M} \setminus \widebar{\mathcal{N}}$. Consider an Iwasawa decomposition of $\exp(X)$ in $G$
\[
\exp(X) =  \bar{n} am ,  \ n \in \widebar{N}, \ a \in A, \ \text{and} \ \ m \in M,
\]
and $am \neq \mathrm{Id} $. Since $ \forall g,h \  \ d(\mathrm{Id}, gh) \leq d(\mathrm{Id}, g) +d(\mathrm{Id}, h)$ we get 
\begin{align*}
 d( \mathrm{Id} , g_t^{-1} \exp( Y ) g_t )  &= d( \mathrm{Id}, h_t^{-1} e^{-t \alpha V_1}  \bar{n}ame^{t \alpha V_1}  h_t^{-1} ) \\
 &\geq d( \mathrm{Id}, h_t^{-1} e^{-t \alpha V_1}ame^{t \alpha V_1}  h_t) -d( \mathrm{Id}, h_t^{-1} e^{-t \alpha V_1}  \bar{n} e^{t \alpha V_1}  h_t) .
 \end{align*}
 Writting $\bar{n}= \exp( Z), Z \in \widebar{\mathcal{N}}$ and $d( \mathrm{Id}, h_t^{-1} e^{-t \alpha V_1}  \bar{n} e^{t \alpha V_1}  h_t )$ is dominated by $ e^{-t \alpha + r(h_t) } \Vert  Z \Vert $  (by Lemma \ref{norms} and $iii)$ of Proposition \ref{dprop}) and converges exponentially fast to zero (recall that by Proposition \ref{convht} $r(h_t) =o(t) $ a.s. ). Thus it remains to prove that $\liminf d( \mathrm{Id}, h_t^{-1} e^{-t \alpha V_1}ame^{t \alpha V_1}  h_t)  >0$ to finish the proof in the first case. This is ensured by the following Lemma.
 
\begin{lem}\label{lem3}
Let $a \in A$ and $m \in M$ s.t $am\neq \mathrm{Id}$. Then $\exists C>0, \ \forall g \in G, \ d( \mathrm{Id},  g^{-1} am g ) > C$.
\end{lem}
\begin{proof}[Proof of lemma \ref{lem3}]
Consider the polar decomposition $g = S(r, \theta ) R$. 
Suppose first that $ a= \mathrm{Id}$ and $m \neq \mathrm{Id}$. Then we get
\begin{align*}
d( \mathrm{Id},  g^{-1} m g )  &= d( \mathrm{Id}, S(r, -\theta ) m S(r, \theta ) ) = d( S(r, \theta) , \  m S(r, \theta )  ) =d( S(r, \theta) ,  S(r, m\theta ) m )   \\
& \geq \frac{\kappa}{ \sqrt{\kappa^{2} +2 \beta^{2}}} d(S(r, \theta), S(r, m\theta) )  \quad  \text{by} \ vi) \ \text{of Proposition \ref{dprop}} \\
& \geq \frac{\kappa}{ \sqrt{\kappa^{2} +2 \beta^{2}}} \left (  d(S(r, \theta), S(r,\theta)m^{-1} ) -d(S(r, \theta)m^{-1} , S(r, m\theta) )\right ) \\
& = \frac{\kappa}{ \sqrt{\kappa^{2} +2 \beta^{2}}} \left (  d( \mathrm{Id}, m^{-1}  ) - d(S(r, \theta), \ m S(r, \theta ) )  \right) \\
&=  \frac{\kappa}{ \sqrt{\kappa^{2} +2 \beta^{2}}} \left (  d( \mathrm{Id}, m ) -d( \mathrm{Id},  g^{-1} m g )   \right) 
\end{align*}
Thus $d( \mathrm{Id},  g^{-1} m g ) \geq \frac{\kappa}{\kappa +\sqrt{\kappa^{2} +2 \beta^{2}}} d( \mathrm{Id}, m ) >0$.

Suppose now $a \neq \mathrm{Id}$. Let $u \neq 0 $ such that $a= \exp ( u V_1)$, then an explicit computation gives:
\begin{align}
\cosh r( g^{-1}a m g) &=  \cosh(u) \left (  \cosh(r)^2 - ((m\theta)^1)^2 \sinh(r)^2 \right ) -\sinh(r)^2 \sum_{i=2}^{d} \theta^i (m\theta)^{i} \\
& = \cosh(u) + \left ( \cosh(u) (1- ((m\theta)^1)^2) - \sum_{i=2}^{d} \theta^i (m\theta)^{i} \right ) \sinh(r)^2  \\
& \geq \cosh(u)  + ( 1 -  \theta^t (m\theta) ) \sinh(r)^2 \quad \text{we used }   (m\theta)^1 = \theta^1 \\
& \geq \cosh(u).
\end{align}
Then by $vi)$ and $vii)$ of Proposition  \ref{dprop} it comes
\[
d( \mathrm{Id}, g^{-1}a m g ) \geq \frac{\kappa}{\sqrt{\kappa^2 + 2 \beta^2 }} \kappa u >0.
\]
\end{proof}
 Return to the proof of Theorem \ref{stable}.

\noindent {\bf Second case:  $\mathbf{X^0=0}$ but $ \mathbf{x^0 \neq 0}$ }. So $\widetilde{X} = ( X^{+}, x^{+} +x^{0})$ and explicitely  
\[
\exp (\widetilde{X}) = ( \exp (X^{+}) , \ x^{0} + x^{+} + \frac{X^{+} x^{0} }{2} ) = (\mathrm{Id},  \xi ) ( \exp (X^{+}), 0 ),
\]
 where we have set $\xi:=  x^{0} + x^{+} + \frac{X^{+} x^{0} }{2}$.

Thus
\begin{align*}
\tilde{g}_t^{-1} \exp(†\widetilde{Y}) \tilde{g}_t &= \tilde{h}_t^{-1} (e^{-t\alpha V_1}, 0) \exp (\widetilde{X} ) (e^{t\alpha V_1}, 0) \tilde{h}_t \\ 
&= ( \mathrm{Id}, h_t^{-1} e^{-t \alpha V_1} \xi ) (\exp( \mathrm{Ad}(h_{t}^{-1} e^{-t\alpha V_1}) X), \ 0  ),
\end{align*}
and
\begin{align*}
d( \mathrm{Id}, &\tilde{g}_t^{-1} \exp(†\widetilde{Y}) \tilde{g}_t) \geq d\left( \mathrm{Id}, \ ( \mathrm{Id}, h_t^{-1} e^{-t \alpha V_1} \xi) \right  )  - d( \mathrm{Id}, \ (\exp( \mathrm{Ad}(h_{t}^{-1} e^{-t\alpha V_1}) X), \ 0  )  ).
\end{align*}
As done previously in the first case, $d( \mathrm{Id}, \ (\exp( \mathrm{Ad}(h_{t}^{-1} e^{-t\alpha V_1}) X), \ 0  )  )$ converges exponentially fast to $0$  and it remains to prove that 
\begin{align} 
\liminf_{t \to \infty} d\left( \mathrm{Id}, \ ( \mathrm{Id}, h_t^{-1} e^{-t \alpha V_1} \xi) \right  )  > 0 \label{liminf2}.
\end{align}
Suppose by contradiction that we can find $s_t$ such that $d\left ( \mathrm{Id} ,  \left  ( \mathrm{Id} , \ h_{s_t}^{-1} e^{-s_{t} \alpha V_1}  (\xi) \right  )\right  )$ converges to $0$. By  $iv)$ of Proposition \ref{dprop}  for large $t$ 
\[
d\left ( \mathrm{Id} ,   ( \mathrm{Id} , \ h_{s_t}^{-1} e^{-s_{t} \alpha V_1}  (\xi ) )\right  ) \geq †\Vert  h_{s_t}^{-1} e^{-s_{t} \alpha V_1}  (\xi ) \Vert
\]

 Since $\frac{X^{+} x^{0} }{2} \in U^{+}$ we obtain directly that $q( \xi ) = q( x^{0})$ which is negative since $x^{0}$ is supposed to be non zero. But 
 \begin{align*}
\Vert  h_{s_t}^{-1} e^{-s_{t} \alpha V_1}  (\xi ) \Vert^{2} &=  \gamma^{2} \left ( \sum_{i =1}^{d} q \left ( h_{s_t}^{-1} e^{-s_{t} \alpha V_1}  (\xi ) , e_i \right )^2  \right )  + \delta^2 q \left ( h_{s_t}^{-1} e^{-s_{t} \alpha V_1}  (\xi ) , e_0 \right )^2   \\
& \geq \min(\gamma, \delta)^2  \left ( \sum_{i =0}^{d} q \left ( h_{s_t}^{-1} e^{-s_{t} \alpha V_1}  (\xi ) , e_i \right )^2  \right ) \\
&= \min(\gamma, \delta)^2  \left ( 2 q( h_{s_t}^{-1} e^{-s_{t} \alpha V_1}  (\xi ) , e_0)^2 - q \left (  h_{s_t}^{-1} e^{-s_{t} \alpha V_1}  (\xi ) \right )^2  \right )\geq -  \min(\gamma, \delta)^2 q( x^{0}) >0. 
 \end{align*} 
 \end{proof}
 
\subsection{Projection on $\Hyp^d \times \R^{1,d}$ of the stable manifolds}
 
 We explicit here the projection of $\mathcal{V}_{\infty}^{-}$ on $\Hyp^{d} \times \R^{1,d}$. Recall that by definition an element of $\mathcal{V}_{\infty}^{-}$ is of the form $\tilde{g}_\infty \exp (X, x) \tilde{g}_\infty ^{-1}$ where  $(X, x) \in \widebar{\mathcal{N}} \times U^{+}$. We deduce, since in this case $ \exp (X, x) = ( \exp(X) , x) $,  that an element of $\pi ( \mathcal{V}_{\infty}^{-} ) $ is of the form
 \begin{align} \label{skew}
\left ( n_\infty \exp(X) n_\infty^{-1} (e_0 ), \ u n_{\infty}(e_0 +e_1)  + \lambda_\infty \left ( \mathrm{Id} - n_\infty \exp(X) n_\infty^{-1} \right  )  (e_0 -e_1) \right ),
 \end{align}
 where $X$ lies in  $\widebar{\mathcal{N} }$ and $u \in \R $.

 Since  $\exp(X)(e_0 + e_1) =e_0 + e_1$ for $X \in \widebar{\mathcal{N}}$ we obtain 
 \[
 q( n_\infty \exp(X) n_\infty^{-1} (e_0), n_\infty (e_0 + e_1) ) = q(n_\infty^{-1} e_0, e_0 +e_1 ) = q( e_0 , n_\infty (e_0 +e_1) )
 \]
 and thus when $X$ describes $\widebar{\mathcal{N}}$ then $ n_\infty \exp(X) n_\infty^{-1} (e_0)$ draws the intersection between $\Hyp^d$ and the affine hyperplan passing by $e_0$ and $q$-orthogonal to $n_\infty (e_0 +e_1)$. This submanifold of $\R^{1,d}$  is a parabolo\"id of codimension 2 and is mapped by $\mathbf{p}$ (the projection onto the projective space) on a sphere tangent at $\partial \Hyp^d$ in $\theta_\infty$ and passing by $\mathbf{p}(e_0)$. It is called the \emph{horosphere} tangent at $\theta_\infty$ and passing by $e_0$ and is denoted by $\mathcal{H}_\infty$. 
 
 Moreover, since 
 \[
 q( n_\infty \exp(X) n_\infty^{-1} (e_0 -e_1), n_\infty (e_0 + e_1) ) = q( e_0 -e_1, e_0 +e_1) = 0,
 \]
 we get that when $X$ describes $\mathcal{N}$ then $n_\infty \exp(X) n_\infty^{-1} (e_0 -e_1)$ describes the intersection between the light cone $\{ \xi , q(\xi) =0 \}$ and the hyperplan passing by $e_0 -e_1$ and $q$-orthogonal to $n_\infty (e_0 +e_1)$. Thus, when $X$ describes $\mathcal{N} $ then $\left ( \mathrm{Id} - n_\infty \exp(X) n_\infty^{-1} \right  )  (e_0 -e_1)$ draws a parabolo\"id $\mathcal{P}_\infty$ in the hyperplan $q$-orthogonal to $n_\infty (e_0 +e_1)$. For each $\dot{\xi}$ in the horosphere $\mathcal{H}_\infty$ corresponds a unique $X_{\dot{\xi}}  \in \mathcal{N}$ such that $\dot{\xi}= n_\infty \exp(X_{\dot{\xi}}) n_\infty^{-1} (e_0)$ and the one-to-one function $ \psi: \dot{\xi} \mapsto \left ( \mathrm{Id} - n_\infty \exp(X_{\dot{\xi}}) n_\infty^{-1} \right  )  (e_0 -e_1)$ maps $\mathcal{H}_{\infty}$ on $\mathcal{P}_{\infty}$. 
 
 Then by \eqref{skew}, we obtain the following one-to-one map 
 \[
 \begin{matrix}
 \mathcal{H}_\infty  \times \langle n_{\infty}(e_0 +e_1) \rangle  & \longrightarrow & \pi( \mathcal{V}_{\infty}^{-}) \\
 (\dot{\xi}, \xi ) & \longmapsto & ( \dot{\xi}, \xi + \lambda_\infty \psi ( \dot{\xi}) )
 \end{matrix}
 \]
 and $\pi( \mathcal{V}_{\infty}^{-})$ is a skew product of the line $\langle n_{\infty}(e_0 +e_1) \rangle$ with the horosphere $\mathcal{H}_\infty$.
 
 \begin{figure}[h]
 \centering
 \scalebox{1} 
{
\begin{pspicture}(0,-3.6111991)(15.1629095,3.6111991)
\definecolor{color16}{rgb}{1.0,0.0,0.2}
\definecolor{color8}{rgb}{0.2,0.2,1.0}
\pscircle[linewidth=0.03,dimen=outer](1.9595605,-0.5791992){1.72}
\psdots[dotsize=0.12](1.9195606,-0.7791992)
\pscircle[linewidth=0.04,linecolor=red,dimen=outer](2.6295605,-0.16919918){0.91}
\psdots[dotsize=0.12](3.3995605,0.3008008)
\usefont{T1}{ptm}{m}{n}
\rput(3.7624707,0.6658008){$\theta_\infty$}
\usefont{T1}{ptm}{m}{n}
\rput(1.7724707,-1.0941992){$p(e_0)$}
\usefont{T1}{ptm}{m}{n}
\rput(1.7624707,2.2458007){$p(\mathcal{H}_\infty)$}
\psline[linewidth=0.04cm,linecolor=color8](7.379561,-3.2591994)(12.95956,2.3208008)
\psline[linewidth=0.024cm](8.95956,-1.6591991)(5.7395606,1.5808008)
\usefont{T1}{ptm}{m}{n}
\rput(8.572471,-1.7541991){$0$}
\usefont{T1}{ptm}{m}{n}
\rput(10.532471,3.0858006){$\langle n_\infty(e_0 +e_1) \rangle$}
\psline[linewidth=0.024cm](7.0995607,-0.87919915)(4.899561,-3.0191994)
\psline[linewidth=0.024cm](4.879561,-3.0591993)(7.5595603,-3.599199)
\psbezier[linewidth=0.038,linecolor=red](8.95956,-1.6591991)(10.53956,-2.099199)(12.099561,-0.5391992)(13.939561,1.3008007)
\psbezier[linewidth=0.02,linecolor=color16,linestyle=dashed,dash=0.16cm 0.16cm](8.979561,-1.6791992)(7.6795607,-1.2191992)(10.899561,2.520801)(11.279561,2.8408008)
\psbezier[linewidth=0.024,linecolor=red,linestyle=dotted,dotsep=0.16cm](7.6795607,-2.9991992)(8.359561,-3.2591994)(9.619561,-3.3991992)(12.659561,-0.85919917)
\psbezier[linewidth=0.024,linecolor=red,linestyle=dotted,dotsep=0.16cm](7.65956,-2.979199)(6.9395604,-2.7391992)(8.12,0.079199165)(9.48,2.2791991)
\psbezier[linewidth=0.024,linestyle=dashed,dash=0.16cm 0.16cm](5.77956,1.5208008)(6.15956,1.2208009)(11.479561,1.1408008)(12.139561,1.5008008)
\psbezier[linewidth=0.024](5.75956,1.5808008)(6.1795607,1.9408008)(11.79956,2.040801)(12.139561,1.5208008)
\psbezier[linewidth=0.04,linecolor=red](10.4195595,1.8608009)(10.559561,2.0808008)(10.979561,2.540801)(11.25956,2.8208008)
\usefont{T1}{ptm}{m}{n}
\rput(13.56247,-1.3741992){$\mathcal{P}_\infty$}
\psline[linewidth=0.024cm,arrowsize=0.05291667cm 2.0,arrowlength=1.4,arrowinset=0.4]{->}(8.95956,-1.6591991)(8.95956,-0.61919916)
\usefont{T1}{ptm}{m}{n}
\rput(9.4824705,-0.3941992){$e_0$}
\psbezier[linewidth=0.02,arrowsize=0.05291667cm 2.0,arrowlength=1.4,arrowinset=0.4]{->}(1.14,1.8991991)(1.14,1.1191992)(1.3,0.73919916)(1.86,0.41919917)
\psbezier[linewidth=0.02,arrowsize=0.05291667cm 2.0,arrowlength=1.4,arrowinset=0.4]{->}(10.22,2.7191992)(10.64,1.9991992)(11.46,1.9791992)(12.32,1.9191992)
\psbezier[linewidth=0.02,arrowsize=0.05291667cm 2.0,arrowlength=1.4,arrowinset=0.4]{->}(13.9,-1.0408008)(14.06,-0.64080083)(14.18,0.019199168)(13.3,0.33919916)
\psbezier[linewidth=0.04,linecolor=red](8.68,-1.4008008)(8.7,-1.4608009)(8.82,-1.6208009)(8.96,-1.6608008)
\psbezier[linewidth=0.024,linecolor=red,linestyle=dotted,dotsep=0.16cm](10.85956,0.20080073)(11.68,-0.20080084)(14.24,2.099199)(15.02,3.1991992)
\psbezier[linewidth=0.024,linecolor=color16,linestyle=dotted,dotsep=0.16cm](10.8395605,0.22080103)(10.11956,0.46080074)(12.26,3.2191992)(12.72,3.599199)
\end{pspicture} 
}
 \caption{ $\pi( \mathcal{V}_{\infty}^{-}) $ is a skew-product of a horosphere with a line}
 \end{figure}
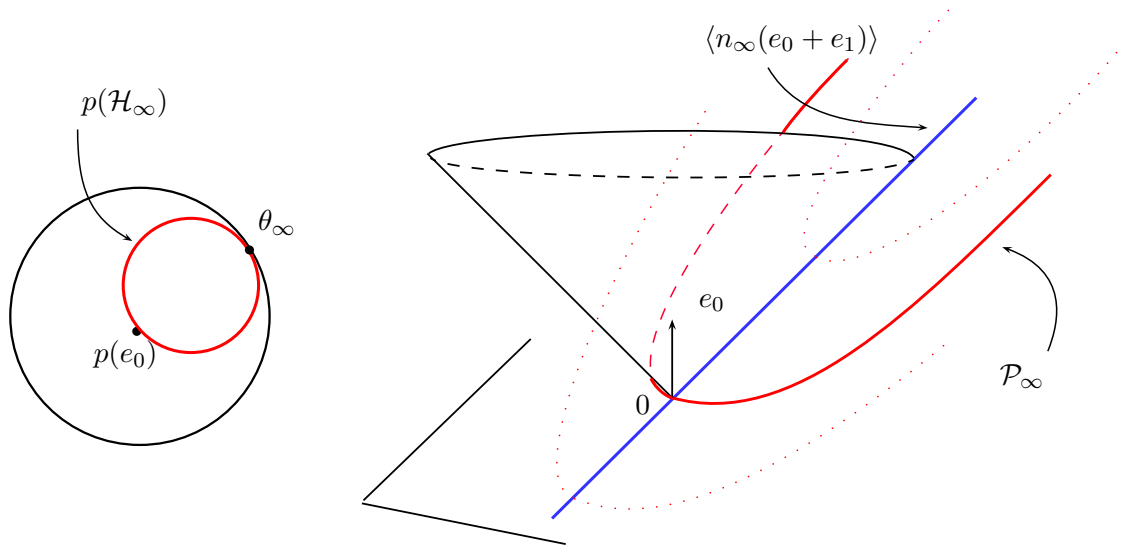
\newpage
\bibliographystyle{plain}
\bibliography{biblio}
\end{document}